\documentclass[11pt]{amsart}

\usepackage[margin=2cm]{geometry}
\usepackage[T1]{fontenc}
\usepackage{mathpazo}
\usepackage{microtype}
\usepackage{graphicx,color}
\usepackage{amsmath,amssymb,amsthm,amsfonts,mathtools,mathrsfs}
\usepackage{pifont}
\newcommand{\cmark}{\ding{51}}%
\newcommand{\xmark}{\ding{55}}%
\newcommand{\qmark}{\textbf{?}}
\usepackage{multirow,makecell}
\usepackage{algorithm,algpseudocode,algorithmicx} % 
\usepackage{tikz,pgfplots}
\pgfplotsset{compat=1.18}
\usepackage{enumitem}
\setlist{nosep}
\setlist[enumerate,1]{label=(\roman*),font=\normalfont}
\usepackage[colorlinks=true, 
            linkcolor=blue, 
            urlcolor=black, 
            bookmarksopen=true]{hyperref}
\usepackage[depth=3]{bookmark}
\usepackage{setspace}
\usepackage{caption,subcaption}
\usepackage{datetime}
\newdateformat{MonthYear}{%
    \monthname[\THEMONTH], \THEYEAR    
}
\DeclareMathOperator{\image}{im}
\DeclareMathOperator{\kernel}{ker}

\DeclareMathOperator{\codim}{codim}
\DeclareMathOperator{\linspan}{span}
\DeclareMathOperator{\clos}{cl}
\DeclareMathOperator{\conv}{conv}
\DeclareMathOperator{\syzygy}{Syz}

\DeclareMathOperator{\interior}{int}

\newcommand{\diff}{\mathop{}\!{\mathrm{d}}}

\newcommand{\nVert}[1]{\lVert#1\rVert}
\newcommand{\sAngle}[1]{\langle#1\rangle}
\newcommand{\pAngle}[2]{\langle#1,#2\rangle}
\newcommand{\transpose}{\mathsf{T}}
\newcommand{\mathcolon}{\mathbin{:}}
\newcommand{\Pythagoras}{\mathrm{py}}
\makeatletter
\def\ubar#1{\underline{\sbox\tw@{$#1$}\dp\tw@\z@\box\tw@}}
\makeatother
\makeatletter
\def\defcal#1{%
\expandafter\newcommand\csname cal#1\endcsname{\mathcal{#1}}}
\count@=0
\loop
\advance\count@ 1
\edef\y{\@Alph\count@}%
\expandafter\defcal\y
\ifnum\count@<26
\repeat
\def\defrm#1{%
\expandafter\newcommand\csname rm#1\endcsname{\mathrm{#1}}}
\count@=0
\loop
\advance\count@ 1
\edef\y{\@Alph\count@}%
\expandafter\defrm\y
\ifnum\count@<26
\repeat
\def\defscr#1{%
\expandafter\newcommand\csname scr#1\endcsname{\mathscr{#1}}}
\count@=0
\loop
\advance\count@ 1
\edef\y{\@Alph\count@}%
\expandafter\defscr\y
\ifnum\count@<26
\repeat
\def\defbf#1{%
\expandafter\newcommand\csname bf#1\endcsname{\mathbf{#1}}}
\count@=0
\loop
\advance\count@ 1
\edef\y{\@Alph\count@}%
\expandafter\defbf\y
\ifnum\count@<26
\repeat
\def\defbb#1{%
\expandafter\newcommand\csname bb#1\endcsname{\mathbb{#1}}}
\count@=0
\loop
\advance\count@ 1
\edef\y{\@Alph\count@}%
\expandafter\defbb\y
\ifnum\count@<26
\repeat
\def\defbf#1{%
\expandafter\newcommand\csname bd#1\endcsname{\mathbold{#1}}}
\count@=0
\loop
\advance\count@ 1
\edef\y{\@alph\count@}%
\expandafter\defbf\y
\ifnum\count@<26
\repeat
\def\defrm#1{%
\expandafter\newcommand\csname rm#1\endcsname{\mathrm{#1}}}
\count@=0
\loop
\advance\count@ 1
\edef\y{\@alph\count@}%
\expandafter\defrm\y
\ifnum\count@<26
\repeat
\makeatother

\newtheorem{theorem}{Theorem}[section]
\newtheorem*{theorem*}{Theorem}
\newtheorem{lemma}[theorem]{Lemma}
\newtheorem{corollary}[theorem]{Corollary}
\newtheorem{proposition}[theorem]{Proposition}

\theoremstyle{remark}
\theoremstyle{remark}

\theoremstyle{definition}

\newtheorem{example}[theorem]{Example}

\begin{document}

\title{Spurious local minima in nonconvex sum-of-squares optimization}
\author{Grigoriy Blekherman}
\address{School of Mathematics, Georgia Institute of Technology, 686 Cherry Street, Atlanta, GA 30332} 
\email{greg@math.gatech.edu}
\author{Rainer Sinn}
\address{Mathematisches Institut, Georg-August-Universität Göttingen, Bunsenstraße 3-5, D-37073 Göttingen} 
\email{rainer.sinn@mathematik.uni-goettingen.de}
\author{Mauricio Velasco}
\address{Centro de Matemática (CMAT), Facultad de Ciencias, Universidad de la República, Iguá 4225 esq.\ Mataojo, C.P.\ 11400, Montevideo, Uruguay}
\email{mvelasco@cmat.edu.uy}
\author{Shixuan Zhang}
\address{Wm Michael Barnes '64 Department of Industrial \& Systems Engineering, Texas A\&M University, 101 Bizell Street, College Station, TX 77843}
\email{shixuan.zhang@tamu.edu}
\date{\MonthYear\today}

\begin{abstract}
    We study spurious second-order stationary points and local minima in a nonconvex low-rank formulation of sum-of-squares optimization on a real variety $X$. We reformulate the problem of finding a spurious local minimum in terms of syzygies of the underlying linear series, and also bring in topological tools to study this problem.
    When the variety $X$ is of minimal degree, there exist spurious second-order stationary points if and only if both the dimension and the codimension of the variety are greater than one, answering a question by Legat, Yuan, and Parrilo.
    Moreover, for surfaces of minimal degree, we provide sufficient conditions to exclude points from being spurious local minima.
    In particular, all second-order stationary points associated with infinite Gram matrices on the Veronese surface, corresponding to ternary quartics, lie on the boundary and can be written as a binary quartic, up to a linear change of coordinates, complementing work by Scheiderer on decompositions of ternary quartics as a sum of three squares.
    For general varieties of higher degree, we give examples and characterizations of spurious second-order stationary points in the interior, together with a restricted path algorithm that avoids such points with controlled step sizes, and numerical experiment results illustrating the empirical successes on plane cubic curves and Veronese varieties.
\end{abstract}

\maketitle

\section{Introduction}
Sums of squares play an important role in both real algebraic geometry and optimization. 
A sum-of-squares representation of a real polynomial guarantees its nonnegativity and thus provides a certified lower bound on its minimum value~\cite{blekherman2012semidefinite}. 
Such representations can generally be found by solving linear matrix inequalities (LMIs), or feasibility of semidefinite programs (SDPs), a task generally solvable in polynomial time up to a prescribed error tolerance, through interior-point methods (IPMs)~\cite{nesterov1994interior}.

From a practical point of view, however, the heavy computational burden of the matrix factorization steps in IPMs limits their applicability and motivates the need for alternative low-rank optimization methods for solving SDPs and LMIs. 
A possible approach, proposed for instance in the celebrated Burer-Monteiro method~\cite{burer2003nonlinear,burer2005local}, is to use a nonconvex reformulation of the optimization problem and to solve it via local (gradient-based) descent methods. 
In principle, this nonconvex approach leads to an easier computation in each iteration, at the expense of possibly losing guaranteed convergence to a global minimum. 
However, a significant amount of recent work~\cite{boumal2016nonconvex,pumir2018smoothed,boumal2020deterministic,cifuentes2021burer,cifuentes2022polynomial} shows that in some cases such formulations lead to nonconvex optimization problems having \emph{no spurious local minima}, that is, problems in which every local minimum is a global minimum and thus can often be solved to global optimality by local descent algorithms. 
Such findings are in line with the practical success of these approaches~\cite{singer2011angular,davenport2016overview,majumdar2020recent}.

The point of departure of this work is a remarkable result by Legat, Yuan, and Parrilo~\cite{legat2023low} who showed that for univariate polynomials we can use a low-rank Burer-Monteiro method 
to find optimal solutions without encountering any spurious local minima in sum-of-squares optimization. This leads to significant speed-ups in computation of large examples. Our aim is to better understand and generalize this phenomenon. We establish general links between the algebraic geometry of real projective varieties and the differential geometry of nonconvex formulations for sum-of-squares optimization problems, and then apply them to specific instances.  
We discover general conditions which guarantee that sets of spurious local minima are small, at least in the interior of the sum-of-squares cone, and prove classification theorems describing them completely on in some cases.

There are two natural ways of generalizing the result of ~\cite{legat2023low}. The first is to look at other instances where nonnegative polynomials are the same as sums of squares. These are classified in terms of \emph{varieties of minimal degree} \cite{blekherman2016sums}. Univariate polynomials correspond to curves of minimal degree, and we examine \emph{surfaces of minimal degree} in detail. This includes the celebrated case of ternary quartics in Hilbert's Theorem \cite{hilbert1888ueber}, as well as 
$2\times 2$ matrices with univariate polynomial entries in the variable $t$, which are positive semidefinite for any value of $t$. This clearly generalizes the univariate case, which corresponds to $1\times 1$ matrices.  For surfaces of minimal degree we do find spurious local minima, but they can lie only on the boundary of the cone of sums-of-squares, which is still friendly for computations. We demonstrate by extensive computational experiments that the Burer-Monteiro method (with rank $3$) scales better than the standard SDP algorithm in these cases. 

The second natural direction is to look at curves of higher degree, the simplest example being cubic curves in the plane. Here we also establish via extensive computations that the Burer-Monteiro method (with rank $3$) scales better. In the negative direction, we use our general framework to show that for large number of variables, the Burer-Monteiro method will encounter spurious local minima, even when we go to a quite large rank.
In the rest of this section, we introduce the notation, present our low-rank formulation, and provide an overview of the main results.

\subsection{Low-rank sum-of-squares formulation}

To describe our results in detail we introduce some notation. 
Let \(\bbP^n\) denote the \(n\)-dimensional projective space, and we denote its points by \([x_0\mathcolon x_1\mathcolon\cdots\mathcolon x_n]\) such that \((x_0,\dots,x_n)\neq(0,\dots,0)\) and \([x_0\mathcolon x_1\mathcolon\cdots\mathcolon x_n]=[cx_0\mathcolon cx_1\mathcolon\cdots\mathcolon cx_n]\) for any \(c\in\bbC\setminus0\).
For a real subvariety \(X\subseteq\bbP^n\), let \(I_X\) be its saturated homogeneous ideal in the polynomial ring \(S:=\bbR[x_0,x_1,\dots,x_n]\) and let \(R:=S/I_X\) denote its homogeneous coordinate ring.
We use \(R_d\) (resp.\ \(S_d\)) to denote the degree-\(d\) homogeneous part of \(R\) (resp.\ \(S\)), e.g., \(R_1\) for all linear forms on \(X\).
By re-embedding \(X\) through the Veronese maps \(\nu_d\) when needed, we focus on the convex cone \(\Sigma_X\) of sums of squares of linear forms on \(X\), defined as
\[
    \Sigma_X:=\{g\in R_2:g=\textstyle\sum_{i=1}^{k}l_i^2 \text{ for some }k\in\bbZ_{\ge 0} \text{ and }l_1,\dots,l_k\in R_1\}.
\]

For instance, the case of univariate polynomials of degree at most $d$ examined in \cite{legat2023low} corresponds to
the Veronese embedding $\nu_d(\bbP^1)\subset \bbP^d$ is the \emph{rational normal curve} parametrized by all homogeneous monomials of degree $d$ in two variables, concretely
\[ \nu_d(\bbP^1) = \{ [y_0^d \mathcolon y_0^{d-1}y_1 \mathcolon \ldots \mathcolon y_1^d] \in \bbP^d \colon [y_0 \mathcolon y_1] \in \bbP^1 \}. \]
The vanishing ideal $I_{\nu_d(\bbP^1)}$ is generated by the $2\times 2$ minors of the matrix
\[ \begin{pmatrix}
    x_0 & x_1 & \ldots & x_{d-1} \\
    x_1 & x_2 & \ldots & x_d
\end{pmatrix}.
\]
The quotient $R = S/ I_{\nu_d(\bbP^1)}$ is a graded ring. Its degree $1$ part $R_1$ is isomorphic to the vector space of homogeneous polynomials of degree $d$ in the variables $s$ and $t$ and, similarly, its degree $2$ part $R_2$ is isomorphic to the vector space of homogeneous polynomials of degree $2d$. In $R_2$, the cone $\Sigma_{\nu_d(\bbP^1)}$ is the cone of sums of squares of forms of degree $2d$.

Given a target quadratic form \(\bar{f}\in R_2\) we wish to find a best approximation for \(\bar{f}\) via sums of squares. 
To this end we fix a norm \(\nVert{\cdot}=\sqrt{\pAngle{\cdot}{\cdot}}\) induced by some inner product on \(R_2\) and consider the optimization problem
\begin{equation}
\label{eq:SOSConvexFormulation}
\min_{g}\left\{\left\|g-\bar{f}\right\|^2: g\in \Sigma_X\right\}.
\end{equation}
Problem~(\ref{eq:SOSConvexFormulation}) is a convex optimization problem with a unique minimizer and its optimal value is equal to zero precisely when $\bar{f}$ is a sum of squares.

Next, we define a nonconvex rank-$k$ reformulation of Problem~(\ref{eq:SOSConvexFormulation}). 
For an auxiliary integer $k$, define the sum-of-squares map $\sigma_k: R_1^k\rightarrow R_2$ which sends $\bdl:=(l_1,\dots, l_k)$ to $\sigma_k(\bdl):=\sum_{i=1}^{k}l_i^2$ and consider the problem
\begin{equation}\label{eq:SOSMinimizationProblem}
    \min_{\bdl}\{\nVert{\sigma_k(\bdl)-\bar{f}}^2:\bdl=(l_1,\dots,l_k)\in R_1^k\}.
\end{equation}
In the above example $X = \nu_d(\bbP^1)$, this is equivalent to the setup for the Burer-Monteiro approach in \cite{legat2023low} via homogenization and dehomogenization of the polynomials.
We refer readers to~\cite{cox2015ideals,blekherman2012semidefinite} for more details.

There are two natural questions for Problem~\eqref{eq:SOSMinimizationProblem}, which is, in general, not convex.
\begin{enumerate}
    \item As the efficiency critically depends on the value of \(k\), how small can we make it so that the optimal values of Problems~\eqref{eq:SOSConvexFormulation} and~\eqref{eq:SOSMinimizationProblem} coincide?
    \item Is there any target \(\bar{f}\) that leads to \emph{spurious local minima}, i.e., local minima of the objective function \(\nVert{\sigma_k(\cdot)-\bar{f}}^2\) that are not global minima? 
    How are they affected by the choice of \(k\)?
\end{enumerate}
The answer to the first question is known as the \emph{Pythagoras number} of \(X\), denoted as \(\Pythagoras(X)\), which is the smallest positive integer \(r\) such that any \(g\in\Sigma_X\) can be written as a sum of \(r\) squares.
Historically, the study of Pythagoras numbers was mostly focused on forms on \(\bbP^n\) and dates back to Hilbert~\cite{hilbert1888ueber}. His most famous result in this context is about ternary quartics, see~\cite{powers2004new,pfister2012elementary} for modern treatments and extensions.
Bounds on Pythagoras numbers for \(\nu_d(\bbP^n)\) (multivariate degree-\(2d\) forms) are studied and improved in~\cite{choi1995sums,scheiderer2017sum,blekherman2024pythagoras}. Thus in this work, we mostly focus on $k$ equal to \(\Pythagoras(X)\).

The second question is more closely related to the study of Burer-Monteiro methods.
From a practical perspective, first- and second-order optimality conditions, i.e., the gradient of the function \(\nVert{\sigma_k(\cdot)-\bar{f}}^2\) being zero and its Hessian matrix being positive semidefinite, are often used as more verifiable necessary conditions for local minimality.
We refer to the points that are not global minima but satisfy the first- and second-order optimality conditions for some target $\bar{f}$ as \emph{spurious second-order stationary points}.
In the geometric language of this work, the unexpectedly interesting result in~\cite{legat2023low} says that there are no spurious second-order stationary points, and consequently no spurious local minima, for the rank-2 formulation on the rational normal curve \(X=\nu_d(\bbP^1)\)
(see a more precise statement in Lemma~\ref{lemma:RationalNormalCurves}).
This motivates us to study the same question for other varieties \(X\), in terms of both spurious second-order stationary points and spurious local minima.

\subsection{Overview of main results}

In this work, we focus on (irreducible and non-degenerate) projective varieties $X\subseteq \bbP^n$ that are \emph{totally real}, i.e., the set of real points \(X(\bbR)\) is Zariski dense in \(X\).
We first consider \emph{varieties of minimal degree}, which are the totally real projective varieties satisfying any (and thus all) of the following equivalent defining properties.
\begin{itemize}
    \item Having the minimal possible degree $\deg(X)=\codim(X)+1$~\cite{eisenbud1987varieties}.
    \item Having the minimal possible Pythagoras numbers $\Pythagoras(X)=\dim(X)+1$~\cite{blekherman2019low,blekherman2021sums}.
    \item Having $\Sigma_X =\rmP_X$ where $\rmP_X$ denotes the set of quadratic forms that are nonnegative on $X(\bbR)$~\cite{blekherman2016sums}.
\end{itemize}
Our first result extends the theorem in~\cite{legat2023low} and shows that rational normal curves are the only interesting case of varieties of minimal degree with no spurious second-order stationary points.

\begin{theorem}\label{thm:ExistenceSpuriousMinima}
    If $X\subseteq \bbP^n$ is a smooth variety of minimal degree, then there are no spurious second-order stationary points for the nonconvex formulation~\eqref{eq:SOSMinimizationProblem} of rank $k=\dim(X)+1$ if and only if \(\dim(X)\in\{1,n-1,n\}\).
\end{theorem}

In particular, this result implies the existence of spurious second-order stationary points for low-rank approximation problems for univariate real symmetric matrix polynomials.
More explicitly, given an \(m\times m\) symmetric matrix \(Q(y)=(q_{ij}(y))_{i,j=1,\dots,m}\) with entries \(q_{ij}(y)\in\bbR[y]\) being univariate polynomials, one may seek a factorization
\begin{equation*}
    Q(y)=L(y)L(y)^\transpose,\quad\text{ for some }L(y)\in(\bbR[y])^{m\times k}
\end{equation*}
to certify its positive semidefiniteness for any real value of \(y\). 
Such factorization problem is equivalent to finding a sum-of-square representation of a nonnegative quadratic form on an \(m\)-dimensional \emph{rational normal scroll} and it is known that \(k=m+1\) always suffices whenever such representation exists, see \cite[Section 2]{blekherman2019low}. 
In the case of $k = m+1 = 3$, we give explicit examples of spurious stationary points of Problem~\eqref{eq:SOSMinimizationProblem} on some rational normal scrolls, see Example~\ref{ex:2DimScrollSpuriousMinima} and Proposition~\ref{prop:ScrollSpuriousMinima}.

From an algorithmic point of view (see Problem~\eqref{eq:SOSRelaxationProblem} and Algorithm~\ref{alg:AvoidSpuriousStationaryPoints} below), it is often helpful to distinguish spurious stationary points or minima \emph{in the interior}, i.e., tuples \(\bdl\) such that \(\sigma_k(\bdl)\) is an interior point of \(\Sigma_X\), from those \emph{on the boundary}, i.e., tuples \(\bdl\) such that \(\sigma_k(\bdl)\) is a boundary point of \(\Sigma_X\).
Our second result is regarding spurious local minima in the interior over surfaces of minimal degree, which all have Pythagoras number $3$.

\begin{theorem}\label{thm:SurfacesOfMinimalDegree}
    Suppose \(X\subseteq\bbP^n\) is a surface of minimal degree.
    If \(\sigma_k(\bdl)\) defines a reduced subscheme on \(X\), then \(\bdl\) cannot be a spurious local minimum for the nonconvex formulation~\eqref{eq:SOSMinimizationProblem} of rank $3$ in the interior.
\end{theorem}

While Theorem~\ref{thm:SurfacesOfMinimalDegree} excludes many points from being spurious local minima on surfaces of minimal degree, a complete analysis of the remaining non-reduced cases can be challenging.
Nevertheless, our third result shows that there is no spurious local minima for the Veronese surface $X=\nu_2(\bbP^2)$ when $k=\Pythagoras(X)=3$.
Moreover, we classify all spurious second-order stationary points whose sum of squares is associated with infinitely many Gram matrices.
This partially extends the previous work by Scheiderer on ternary quartics~\cite{scheiderer2017sum}. 
\begin{theorem}\label{thm:VeroneseSurfaceMinima}
    When \(X=\nu_2(\bbP^2)\subseteq\bbP^5\) is the Veronese surface,
    \begin{itemize}
        \item there is no spurious second-order stationary point of~\eqref{eq:SOSMinimizationProblem} in the interior for any $k\ge3$;
        \item when $k=3$, there is no spurious local minimum of~\eqref{eq:SOSMinimizationProblem}; in particular, any spurious second-order stationary points with infinitely many Gram representations must lie on the boundary, and correspond to a binary quartic form on \(\bbP^2\) (up to a projective change of coordinates).
    \end{itemize}
\end{theorem}

We summarize Theorems~\ref{thm:ExistenceSpuriousMinima}, \ref{thm:SurfacesOfMinimalDegree}, and~\ref{thm:VeroneseSurfaceMinima} in Table~\ref{tab:VarietiesOfMinimalDegree}.
\begin{table}[htbp]
    \centering
    \begin{tabular}{ccccc}
    \hline
    & \multicolumn{2}{c}{boundary} & \multicolumn{2}{c}{interior} \\
    \cline{2-5}
    Variety of minimal degree & \makecell{spurious\\ local minima} & \makecell{spurious 2nd-order\\ stationary points} & \makecell{spurious\\ local minima} & \makecell{spurious 2nd-order\\ stationary points}\\
    \hline
    rational normal curves & \xmark & \xmark & \xmark & \xmark\\
    quadric hypersurfaces or $\bbP^n$ & \xmark & \xmark & \xmark & \xmark\\
    Veronese surface $\nu_2(\bbP^2)$ & \xmark & \cmark & \xmark & \xmark\\
    rational normal scrolls & \cmark & \cmark & \qmark & \qmark\\
    \hline
    \end{tabular}
    \caption{Existence of spurious local minima and second-order stationary points on varieties of minimal degree}
    \label{tab:VarietiesOfMinimalDegree}
\end{table}
The basic technical tool used to prove these results is the following intrinsic characterization of spurious local minima on varieties, which relates to the quadratic part of the ideal \(\sAngle{\bdl}_2=\{\sum_{i=1}^{k}l_ih_i:h_1,\dots,h_k\in R_1\}\), and the linear part of the syzygies \(\syzygy_1(\bdl):=\{(h_1,\dots,h_k)\in R_1^k:\sum_{i=1}^{k}l_ih_i=0\}\) of a given tuple \(\bdl\in R_1^k\),
corresponding to \(\sAngle{\bdl}_2 = \image{\diff_\bdl\sigma_k}\) and \(\syzygy_1(\bdl)=\kernel{\diff_\bdl\sigma_k}\), the image and the kernel of the differential of the sum-of-squares map at $\bdl$, respectively. 
Fix an inner product on $R_2$; we use $\sAngle{\bdl}_2^\perp$ to denote the orthogonal subspace of $\sAngle{\bdl}_2$ (or we write $g\perp\sAngle{\bdl}_2$ for any $g\in\sAngle{\bdl}_2^\perp$), and $\Sigma_X^*$ to denote the convex dual cone of $\Sigma_X$.
The following result holds for arbitrary varieties $X\subset \bbP^n$ and any inner product on $R_2$.
\begin{theorem}\label{thm:CharacterizationSpuriousMinima}
    Fix a $k$-tuple \(\bdl=(l_1,\dots l_k)\in R_1^k\).
    The following conditions are equivalent for the nonconvex formulation~\eqref{eq:SOSMinimizationProblem} of rank $k$.
    \begin{enumerate}
        \item There exists a target \(\bar{f}\in R_2\) such that \(\bdl\) is a spurious second-order stationary point.
        \item There exists \(g\in R_2\) such that \(g\perp\sAngle{\bdl}_2\), \(g\notin\Sigma_X^*\), and for any \(\bdh\in\syzygy_1(\bdl)\), \(\pAngle{g}{\sigma_k(\bdh)}\ge0\); equality can only hold here for $\bdh \in \syzygy_1(\bdl)$ if \(g\perp\sAngle{\bdh}_2\).
    \end{enumerate}
    Moreover, when the ideal \(\sAngle{\bdl}\subset R\) is real radical, 
    both conditions~\textnormal{(i)} and \textnormal{(ii)} are equivalent to the following:
    \begin{enumerate}
        \item[(iii)] There exists a target \(\bar{f}\in R_2\) such that \(\bdl\) is a spurious local minimum.
    \end{enumerate}
\end{theorem}

Next we expand our scope to varieties of higher degree.
Example~\ref{ex:QuarticsSpuriousMinima} shows that 
on a general variety $X\subseteq\bbP^n$, there may exist spurious local minima in the interior of $\Sigma_X$, and that this behaviour persists even when $k$ is larger than any prescribed fraction of $n+1$.
We therefore focus on bounding the size of the locus of spurious second-order stationary points in the interior.

One major difficulty is the determination of the Pythagoras numbers for these varieties.
While this is a quite difficult task in general, recent work~\cite{blekherman2021sums} has shown some upper bounds on the Pythagoras number. 
For instance, \cite[Theorem 2.2]{blekherman2021sums} says that if $r(X)$ is the smallest integer $k$ such that any $k$ linearly independent linear forms with no common zero on $X$ generate \(R_2\), then \(\Pythagoras(X)\le r(X)\). 
Our next theorem shows that for such $k$ the locus of spurious first-order stationary points in the interior of \(\Sigma_X\) is small.

\begin{theorem} \label{thm:SmallLocusSpuriousMinima} 
    Suppose $X\subseteq \bbP^n$ is a smooth, totally real variety. 
    If $k\geq r(X)$ then the Zariski closure of the set of $\sigma_k(\bdl)$ as $\bdl$ ranges over the spurious first-order stationary points for some $\bar{f}\in R_2$, such that  $\sigma_k(\bdl)$ is interior to $\Sigma_X$ has codimension at least two.
\end{theorem}

Our final contribution stems from a practical consideration. 
In polynomial optimization, people are often interested in a general form of the sum-of-\(k\)-squares problem :
 \begin{equation}\label{eq:SOSRelaxationProblem}
    v^*:=\max\{\ubar{v}\le v\le\bar{v}:f-v\cdot g=\sigma_k(\bdl),\text{ for some }\bdl\in R_1^k\},
\end{equation}
where \(f,g\in R_2\), \(\ubar{v}\in\bbR\) and \(\bar{v}\in\bbR\cup\{+\infty\}\).
Does the characterization in Theorem~\ref{thm:SmallLocusSpuriousMinima}
allow us to use local descent methods for finding a global optimum in~\eqref{eq:SOSRelaxationProblem}? 
We answer this question affirmatively by proposing a \emph{restricted path} algorithm for the nonconvex low-rank formulation~\eqref{eq:SOSMinimizationProblem} in Algorithm~\ref{alg:AvoidSpuriousStationaryPoints}.
That is, instead of allowing an arbitrary path for the local descent method, one can restrict each iteration close to the line \(f+\bbR g\) using intermediate targets along this path.
In this way, spurious stationary points on the boundary can be avoided until the search gets close to the optimum \(v^*\), while those in the interior can be avoided as generically the line \(f+\bbR g\) does not intersect with the locus of stationary points asserted by Theorem~\ref{thm:SmallLocusSpuriousMinima}.
Extensive numerical experiments of the formulation~\eqref{eq:SOSRelaxationProblem} are provided on both varieties of minimal degree and other varieties at the end of the paper, to illustrate successes of avoiding spurious stationary points in the interior empirically.

\section{Characterizations of spurious second-order stationary points and local minima}

In this section, we first review the optimality conditions for the rank-\(k\) formulation~\eqref{eq:SOSMinimizationProblem}.
As we are mostly interested in the existence of spurious second-order stationary points and spurious local minima, we then unfix the target \(\bar{f}\) and study the differential information of the objective function in~\eqref{eq:SOSMinimizationProblem} at a given point \(\bdl=(l_1,\dots,l_k)\) with varying \(\bar{f}\).
In particular, we define a cone of \emph{reachable directions} which leads to a necessary condition for \(\bdl\) being a spurious second-order stationary point, and then identify a sufficient condition for \(\bdl\) being a spurious local minimum.
We then prove Theorem~\ref{thm:CharacterizationSpuriousMinima} and mention some of its useful consequences at the end of this section.

We begin with the restating the optimality conditions in our setting.
\begin{lemma}\label{lemma:OptimalityConditions}
    Let \(\bdl=(l_1,\dots,l_k)\in R_1^k\) and \(\bar{f}\in R_2\).
    Then the optimality conditions can be written as follows:
    \begin{itemize}
        \item (first-order) \(\pAngle{\sigma_k(\bdl)-\bar{f}}{\sum_{i=1}^{k}l_ih_i}=0\) for any \(\bdh=(h_1,\dots,h_k)\in R_1^k\); and
    \item (second-order) in addition to the first-order condition, the quadratic map \(\bdh\mapsto 4\nVert{\sum_{i=1}^{k}l_ih_i}^2+2\pAngle{\sigma_k(\bdl)-\bar{f}}{\sigma_k(\bdh)}\) is positive semidefinite on \(R_1^k\).
    \end{itemize}
\end{lemma}
\begin{proof}
    We can calculate the gradient and Hessian for the objective function of~\eqref{eq:SOSMinimizationProblem} by taking an small perturbation \(\varepsilon>0\) and noting
    \begin{equation*}
        \begin{aligned}
            \nVert{\sigma_k(\bdl+\varepsilon\bdh)-\bar{f}}^2-\nVert{\sigma_k(\bdl)-\bar{f}}^2=&\  
    4\pAngle{\sigma_k(\bdl)-\bar{f}}{\textstyle\sum_{i=1}^kl_ih_i}\cdot\varepsilon 
    +\bigl(4\nVert{\textstyle\sum_{i=1}^kl_ih_i}^2+2\pAngle{\sigma_k(\bdl)-\bar{f}}{\sigma_k(\bdh)}\bigr)\cdot\varepsilon^2\\
    &+4\pAngle{\sigma_k(\bdh)}{\textstyle\sum_{i=1}^kl_ih_i}\cdot\varepsilon^3
    +\nVert{\sigma_k(\bdh)}^2\cdot\varepsilon^4.
        \end{aligned}
    \end{equation*}
    This shows that gradient is the map \((h_1,\dots,h_k)\mapsto 4\pAngle{\sigma_k(\bdl)-\bar{f}}{\textstyle\sum_{i=1}^kl_ih_i}\) and the Hessian matrix is the map \(\bdh=(h_1,\dots,h_k)\mapsto8\nVert{\textstyle\sum_{i=1}^kl_ih_i}^2+4\pAngle{\sigma_k(\bdl)-\bar{f}}{\sigma_k(\bdh)}\).
    The remainder of the proof follows directly from the definitions of the first and second-order optimality conditions.
\end{proof}

The optimality conditions depend on the target \(\bar{f}\), which is inconvenient for studying the existence questions.
Given \(\bdl\in R_1^k\), we want to check which target \(\bar{f}\) would make \(\bdl\) satisfy the optimality conditions, and whether it would make \(\bdl\) a spurious local minimum.
Thus we take a closer look at the differential information at \(\bdl\) as follows.
Recall that \emph{normal cone} of \(\Sigma_X\) at the point \(f\in\Sigma_X\) can be defined as
\begin{equation}
    \calN_{f}(\Sigma_X):=\{g\in R_2: \pAngle{g}{f'-f}\le 0\text{ for all }f'\in\Sigma_X\},
\end{equation}
and the \emph{tangent cone} at the same point as
\begin{equation}\label{eq:DefTangentCone}
    \calT_{f}(\Sigma_X):=\clos\{g\in R_2: f+\delta g\in\Sigma_X\text{ for some }\delta >0\}.
\end{equation}
We simply write \(\calN_f\) for \(\calN_f(\Sigma_X)\), and \(\calT_f\) for \(\calT_f(\Sigma_X)\) when there is no confusion about \(\Sigma_X\).
By definition, \(\calN_{f}\) and \(\calT_{f}\) form a pair of polar cones. 
These two cones have explicit descriptions in our case.
\begin{lemma}\label{lemma:NormalConesCharacterization}
    Given any \(\bdl=(l_1,\dots,l_k)\in R_1^k\), let \(\sAngle{\bdl}_2\subseteq R_2\) denote the degree-2 part of the ideal generated by \(l_1,\dots,l_k\).
    Then \(\calN_{\sigma_k(\bdl)}=-\Sigma_X^*\cap\sAngle{\bdl}_2^\perp\) and \(\calT_{\sigma_k(\bdl)}=\clos(\Sigma_X+\sAngle{\bdl}_2)\).
\end{lemma}
\begin{proof}
    For the inclusion \(-\Sigma_X^*\cap\sAngle{\bdl}_2^\perp\subseteq \calN_{\sigma_k(\bdl)}\), note that for any \(g\in-\Sigma_X^*\cap\sAngle{\bdl}_2^\perp\), we have \(\pAngle{g}{f-\sigma_k(\bdl)}=\pAngle{g}{\diff_\bdl\sigma_k(\bdh)}+\pAngle{g}{\sigma_k(\bdh)}\le0\) for any \(f=\sigma_k(\bdl+\bdh)\in\Sigma_X\), \(\bdh=(h_1,\dots,h_k)\in R_1^k\).
    For the other inclusion \(\calN_{\sigma_k(\bdl)}\subseteq -\Sigma_X^*\cap\sAngle{\bdl}_2^\perp\), take any \(g\in \calN_{\sigma_k(\bdl)}\).
    By definition we see that \(\bdh=0\) is a maximum of the function \(G(\bdh):=\pAngle{g}{\sum_{i=1}^{k}l_ih_i}+\pAngle{g}{\sigma_k(\bdh)}\).
    The optimality condition implies that \(\pAngle{g}{\sum_{i=1}^{k}l_ih_i}=0\) for any \(\bdh\in R_1^k\), so \(g\in\sAngle{\bdl}_2^\perp\).
    Thus we see that \(G(\bdh)=\pAngle{g}{\sigma_k(\bdh)}\le 0\), which shows that \(g\in -\Sigma_X^*\).
    Finally it follows from the polarity relation with \(\calN_{\sigma_k(\bdl)}\) that \(\calT_{\sigma_k(\bdl)}=\clos(\Sigma_X+\sAngle{\bdl}_2)\).
\end{proof}

We note that the tangent cone \(\calT_{\sigma_k(\bdl)}(\Sigma_X)\) at the sum of squares \(f=\sigma_k(\bdl)\) in definition~\eqref{eq:DefTangentCone} does not depend on a particular sum of squares representation of $f$.
It is possible that some directions in the tangent cone cannot be reached through local perturbation away from \(\bdl\).
This is a situation where second-order stationary points may occur in~\eqref{eq:SOSMinimizationProblem}, motivated by which we define a \emph{cone of (linearly) reachable directions} at \(\bdl\in R_1^k\) as 
\begin{equation}
    \calR_\bdl(\Sigma_X):=\sigma_k(\kernel{\diff_\bdl\sigma_k})+\image{\diff_\bdl\sigma_k}=\sigma_k(\syzygy_1(\bdl))+\sAngle{\bdl}_2.
\end{equation}
As before, we write \(\calR_\bdl\) for \(\calR_\bdl(\Sigma_X)\) if no confusion is caused.
To justify its name, we note that for any \(f=\sigma_k(\bdl')+\diff_\bdl\sigma_k(\bdl'')\in\calR_\bdl\), for some \(\bdl'\in\syzygy_1(\bdl)\) and \(\bdl''\in R_1^k\), we can construct a curve \(\gamma:(-1,1)\to R_1^k\) defined by \(\gamma(z)=\bdl+\bdl'\cdot\sqrt{z}+\bdl''\cdot z\), such that
\begin{equation}\label{eq:LocallyParametrizedTangentCone}
    \gamma(0)=\sigma_k(\bdl)
    \text{ and }
    \frac{\diff}{\diff z}(\sigma_k\circ\gamma)\vert_{z=0}=\frac{\diff}{\diff z}\sum_{i=1}^k[l_i^2+((l'_i)^2+2l_il''_i)z+o(z)]\big\vert_{z=0}=f.
\end{equation}
Therefore, it is clear that \(\calR_\bdl\subseteq \calT_{\sigma_k(\bdl)}\) and thus so is its closed convex hull \(\clos\conv\calR_\bdl\subseteq \calT_{\sigma_k(\bdl)}\) for any \(\bdl\in R_1^k\).
The containment may be strict:
note that the dual cone of \(\calR_\bdl\) consists of perturbation directions \(g\perp\sAngle{\bdl}_2\) such that \(\bdh\mapsto\pAngle{g}{\sigma_k(\bdh)}\) is positive semidefinite.
Condition (ii) in Theorem~\ref{thm:CharacterizationSpuriousMinima} thus implies that the closed convex hull \(\clos\conv\calR_\bdl\) is strictly contained in \(\calT_{\sigma_k(\bdl)}\).
In fact, they are equivalent when the ideal \(\sAngle{\bdl}\subset R\) is real radical as discussed below.

\begin{lemma}\label{lemma:RealRadicalTangentConeCondition}
    Let \(\bdl=(l_1,\dots,l_k)\in R_1^k\) such that the ideal \(\sAngle{\bdl}\subset R\) is real radical.
    Then \(\clos\conv\calR_\bdl\subsetneq \calT_{\sigma_k(\bdl)}\) if and only if Condition~\textnormal{(ii)} in Theorem~\ref{thm:CharacterizationSpuriousMinima} holds, i.e., there exists \(g\in R_2\) such that \(g\perp\sAngle{\bdl}_2\), \(g\notin\Sigma_X^*\), and for any \(\bdh\in\syzygy_1(\bdl)\), \(\pAngle{g}{\sigma_k(\bdh)}\ge0\), where the equality may hold only if \(g\perp\sAngle{\bdh}_2\).
\end{lemma}
\begin{proof}
    The sufficiency is clear as such \(g\) is in the dual cone of \(\calR_\bdl\), but \(g\notin-\calN_{\sigma_k(\bdl)}=\Sigma_X^*\cap\sAngle{\bdl}_2^\perp\) by Lemma~\ref{lemma:NormalConesCharacterization}.
    Next we show that it is a necessary condition. 
    Since \(\sAngle{\bdl}\) is real radical, it is known that \(\Sigma_X/\sAngle{\bdl}_2\) is a pointed cone~\cite[Lemma 2.1]{blekherman2016sums}.
    Thus \(\sAngle{\bdl}_2\) is the lineality space of \(\Sigma_X+\sAngle{\bdl}_2\) and thus also the lineality space of \(\clos\conv\calR_\bdl\).
    By assumption, we can take a direction \(g\) in the relative interior of \(\calR_\bdl^*\) that is not in \(\calT_{\sigma_k(\bdl)}^*=\Sigma_X^*\cap\sAngle{\bdl}_2^\perp\).
    This means that for any \(\bdh=(h_1,\dots,h_k)\in\syzygy_1(\bdl)\), \(\pAngle{g}{\sigma_k(\bdh)}\ge0\), and if \(\pAngle{g}{\sigma_k(\bdh)}=0\), then \(\sigma_k(\bdh)\) is in the lineality space \(\sAngle{\bdl}_2\).
    By the real radicalness of \(\sAngle{\bdl}\), we see that \(h_1,\dots,h_k\in\sAngle{\bdl}_1\), and consequently \(g\perp\sAngle{\bdh}_2\subseteq\sAngle{\bdl}_2\).
\end{proof}

Unlike second-order stationary points, being a local minimum is  more subtle as we may need to look at higher-order differentials.
One way is to consider any curve emanating from the tuple of linear forms \(\bdl\) and parametrized by a power series
\begin{equation}\label{eq:PowerSeriesExpansion}
    \gamma(z):=\sum_{s=0}^{\infty}z^s\bdl^{(s)}
\end{equation}
for some \(\bdl^{(s)}\in R_1^k\), for each \(s\in\bbZ_{\ge1}\), and \(\bdl^{(0)}:=\bdl\), which gives
\begin{equation}\label{eq:PowerSeriesSOS}
    \sigma_k(\gamma(z))=\sigma_k(\bdl)+\sum_{s=1}^{\infty}z^s\cdot\sum_{t=0}^{s}\sum_{i=1}^{k}l_i^{(t)}l_i^{(s-t)}.
\end{equation}
Let \(f:=\sigma_k(\bdl)-\bar{f}\), and the associated objective function in~\eqref{eq:SOSMinimizationProblem}, denoted as \(\Phi_\gamma\), becomes
\begin{equation}\label{eq:PowerSeriesObj}
    \Phi_\gamma(z)=\nVert{f}^2+\sum_{s=1}^{\infty}z^s\Bigl(2\pAngle{f}{\sum_{t=0}^{s}\sum_{i=1}^{k}l_i^{(t)}l_i^{(s-t)}}+\sum_{r=1}^{s-1}\pAngle{\sum_{p=0}^{r}\sum_{i=1}^{k}l_i^{(p)}l_i^{(r-p)}}{\sum_{q=0}^{s-r}\sum_{j=1}^{k}l_j^{(q)}l_j^{(s-r-q)}}\Bigr).
\end{equation}
For instance, the expansion up to \(s=4\) can be written explicitly as
\begin{equation}\label{eq:PowerSeriesDeg4}
    \begin{aligned}
    \Phi_\gamma(z)&=\nVert{f}^2+z\cdot4\pAngle{f}{\sum_i l_il_i^{(1)}}+z^2\left(2\pAngle{f}{\sum_i 2l_il_i^{(2)}+[l_i^{(1)}]^2}+4\nVert{\sum_i l_il_i^{(1)}}^2\right)\\
    +&z^3\left(4\pAngle{f}{\sum_i l_il_i^{(3)}+l_i^{(1)}l_i^{(2)}}+4\pAngle{\sum_i l_il_i^{(1)}}{\sum_i 2l_il_i^{(2)}+\sum_i[l_i^{(1)}]^2}\right)\\
    +&z^4\left(2\pAngle{f}{\sum_i 2l_il_i^{(4)}+2l_i^{(1)}l_i^{(3)}+[l_i^{(2)}]^2}+8\pAngle{\sum_i l_il_i^{(1)}}{\sum_i l_il_i^{(3)}+l_i^{(1)}l_i^{(2)}}+\nVert{\sum_i 2l_il_i^{(2)}+[l_i^{(1)}]^2}^2\right)+\cdots.
    \end{aligned}
\end{equation}
When the difference \(f\) satisfies the first-order condition, then \(\pAngle{f}{\sum_il_il_i^{(t)}}=0\) for any \(\bdl^{(t)}\in R_1^k\).
If \(f\) further satisfies the second-order necessary condition and is sufficiently small, such that the kernel of the Hessian \(\bdl^{(1)}\mapsto 2\pAngle{f}{\sigma_k(\bdl^{(1)})}+4\nVert{\sum_il_il_i^{(1)}}^2\) is contained in \(\syzygy_1(\bdl)\), then either the second-order term is positive or \(\bdl^{(1)}\in\syzygy_1(\bdl)\).
In the latter case, it is possible that the third- or fourth-order coefficient in~\eqref{eq:PowerSeriesDeg4} becomes negative, which leads to the existence of spurious second-order stationary points that are not local minima.
Nevertheless, we identify the following sufficient condition for local minimality.
\begin{lemma}\label{lemma:LocalMinimumSufficientCondition}
    Let \(\bdl=(l_1,\dots,l_k)\in R_1^k\) be a spurious second-order stationary point for some target \(\bar{f}\in R_2\).
    If for any \(\bdh=(h_1,\dots,h_k)\in\syzygy_1(\bdl)\), the condition \(\pAngle{\sigma_k(\bdl)-\bar{f}}{\sigma_k(\bdh)}=0\) implies that \(h_i\in\sAngle{\bdl}_1\) for each \(i=1,\dots,k\), then \(\bdl\) is a spurious local minimum for some target in $R_2$.
\end{lemma}
\begin{proof}
    By scaling \(f:=\sigma_k(\bdl)-\bar{f}\), we may assume that the kernel of the Hessian is contained in \(\syzygy_1(\bdl)\), without loss of generality. 
    Assume for contradiction that \(\bdl\) is not a local minimum.
    Note that the level set \(\Lambda:=\{\bdh\in R_1^k:\nVert{\sigma_k(\bdh)-\bar{f}}^2<\nVert{\sigma_k(\bdl)-\bar{f}}^2\}\) is a basic semialgebraic set with \(\bdl\in\clos\Lambda\) by the continuity of the objective function.
    Thus by the Nash curve selection lemma~\cite[Proposition 8.1.13]{bochnak2013real}, there exists a Nash mapping \(\gamma:(-1,1)\to R_1^k\) such that \(\gamma(0)=\bdl\) and \(\gamma(z)\in\Lambda\) for any \(0<z<1\).
    Thus with a power series representation of \(\gamma\) (as in~\eqref{eq:PowerSeriesExpansion}), the associated objective function \(\Phi_\gamma\) (in~\eqref{eq:PowerSeriesObj}) must have a negative initial coefficient (i.e., the coefficient of the smallest-degree nonzero term).
    To simplify the notation for \(\Phi_\gamma\), let \(g^{(s)}:=\sum_{t=0}^{s}\sum_{i=1}^{k}l_i^{(t)}l_i^{(s-t)}\in R_2\) denote the degree-\(s\) coefficient in the expansion of the sum-of-squares curve~\eqref{eq:PowerSeriesSOS} for each \(s\ge1\), so that the power series~\eqref{eq:PowerSeriesObj} can be written as
    \begin{equation*}
        \Phi_\gamma(z)=\nVert{f}^2+\sum_{s=1}^{\infty}\left(2\pAngle{f}{g^{(s)}}+\sum_{r=1}^{s-1}\pAngle{g^{(r)}}{g^{(s-r)}}\right) z^s.
    \end{equation*}
    We claim that by our assumption, if the coefficients up to degree \(2t\) are all zero in the series~\eqref{eq:PowerSeriesObj} for some \(t\ge1\), then 
    \begin{enumerate}
        \item \(l_i^{(s)}\in\sAngle{\bdl}_1\) for any \(i=1,\dots,k\) and \(s=1,\dots,t\);
        \item \(g_i^{(s)}=0\) for \(s=1,\dots,t\);
        \item the degree-\((2t+1)\) coefficient is zero; and
        \item the degree-\((2t+2)\) coefficient is either positive or zero.
    \end{enumerate}
    The claim gives the desired contradiction as the initial coefficient must be positive.
    To show the claim, we argue by induction on \(t\) as follows.
    For \(t=1\), this is clear as \(f\) satisfies the second-order optimality condition with the kernel of the Hessian contained in \(\syzygy_1(\bdl)\) by construction.
    Suppose the claim is true for some \(t\ge1\).
    Then the degree-\((2t+1)\) coefficient in~\eqref{eq:PowerSeriesObj} becomes
    \begin{equation*}
        2\pAngle{f}{g^{(2t+1)}}+\sum_{s=1}^{2t}\pAngle{g^{(s)}}{g^{(2t+1-s)}}=0,
    \end{equation*}
    because \(g^{(2t+1)}\in\sAngle{\bdl}_2\perp f\), and \(g^{(s)}=0\) for \(s=1,\dots,t\).
    The degree-\((2t+2)\) coefficient in~\eqref{eq:PowerSeriesObj} can be written as
    \begin{equation*}
        2\pAngle{f}{g^{(2t+2)}}+\sum_{s=1}^{2t+1}\pAngle{g^{(s)}}{g^{(2t+2-s)}}=2\pAngle{f}{\sigma_k(\bdl^{(t+1)})}+\nVert{g^{(t+1)}}^2.
    \end{equation*}
    Here,
    \begin{equation*}
        g^{(t+1)}=\sum_{s=0}^{t+1}\sum_{i=1}^{k}l_i^{(s)}l_i^{(t+1-s)}=2\sum_{i=1}^{k}l_il_i^{(t+1)}+\sum_{r=1}^{t}\sum_{i=1}^{k}l_i^{(r)}l_i^{(t+1-r)},
    \end{equation*}
    so by the induction hypothesis \(\sum_{r=1}^{t}\sum_{i=1}^{k}l_i^{(r)}l_i^{(t+1-r)}\in\sAngle{\bdl}^2\).
    Consequently, there exist \(h_i^{(t+1)}\in\sAngle{\bdl}_1\), \(i=1,\dots,k\), such that \(g^{(t+1)}=\sum_{i=1}^{k}2l_i(l_i^{(t+1)}+h_i^{(t+1)})\).
    As \(f\perp\sAngle{\bdl}_2\), denoting \(\bdh^{(t+1)}:=(h_1^{(t+1)},\dots,h_k^{(t+1)})\in R_1^k\), we have
    \begin{equation*}
        2\pAngle{f}{g^{(2t+2)}}+\sum_{s=1}^{2t+1}\pAngle{g^{(s)}}{g^{(2t+2-s)}}=2\pAngle{f}{\sigma_k(\bdl^{(t+1)}+\bdh^{(t+1)})}+\nVert{\sum_{i=1}^{k}2l_i(l_i^{(t+1)}+h_i^{(t+1)})}^2.
    \end{equation*}
    Now by the construction of \(f\), this degree-\((2t+2)\) coefficient is either positive, or zero which happens only when \(\pAngle{f}{\sigma_k(\bdl^{(t+1)}+\bdh^{(t+1)})}=0\) and \(\bdl^{(t+1)}+\bdh^{(t+1)}\in\syzygy_1(\bdl)\), and this ensures that \(l_i^{(t+1)}\in\sAngle{\bdl}_1\) for each \(i=1,\dots,k\), and \(g^{(t+1)}=0\) by the definition of \(\bdh^{(t+1)}\), completing the induction step.
\end{proof}

\begin{proof}[Proof of Theorem~\ref{thm:CharacterizationSpuriousMinima}]
    To show (i)$\implies$(ii), take a target \(\bar{f}\in R_2\) such that \(\bdl\) is a spurious second-order stationary point, and let \(g:=\sigma_k(\bdl)-\bar{f}\).
    This means that \(\bdl\) is not a global minimum, so \(-g\notin \calN_{\sigma_k(\bdl)}=-\Sigma_X^*\cap\sAngle{\bdl}_2^\perp\) by Lemma~\ref{lemma:NormalConesCharacterization}.
    By the first-order stationary condition, we have \(\pAngle{g}{\diff_\bdl\sigma_k(\bdh)}=0\) for any \(\bdh\in R_1^k\), so \(g\in\sAngle{\bdl}_2^\perp\).
    Moreover, by the second-order stationary condition, for any \(\bdh\in\syzygy_1(\bdl)\), we see that \(\pAngle{g}{\sigma_k(\bdh)}\ge0\).
    Here, equality holds only if the differential of the quadratic map \(\bdh'\mapsto\pAngle{g}{\sigma_k(\bdh')}\) is zero at \(\bdh\), which means that \(\bdh\) is in the kernel of this quadratic map and hence \(g\perp\sAngle{\bdh}_2\).
    Thus \(g\) is the desired direction in (ii).
    
    Next we show (ii)$\implies$(i).
    By Lemma~\ref{lemma:NormalConesCharacterization}, the chosen \(g\) satisfies \(g\in [-(\sigma_k(\syzygy_1(\bdl)))^*\cap\sAngle{\bdl}_2^\perp]\setminus \calN_{\sigma_k(\bdl)}\).
    Note that if we let \(\bar{f}:=\sigma_k(\bdl)-\epsilon g\) for some \(\epsilon\in\bbR_{>0}\), then \(\bdl\) is a first-order stationary point of~\eqref{eq:SOSMinimizationProblem}, and the Hessian can be written as \(\phi+\epsilon\psi\), where \(\phi(\bdh):=4\nVert{\sum_{i=1}^{k}l_ih_i}^2\) and \(\psi(\bdh):=2\pAngle{g}{\sigma_k(\bdh)}\).
    By assumption, \(\psi\) is positive semidefinite on \(\kernel{\phi}=\syzygy_1(\bdl)\) and the zeros of \(\psi\) in \(\kernel{\phi}\) are contained in its kernel \(\kernel{\psi}\).
    By taking the quotient of the common kernel \(\kernel{\phi}\cap\kernel{\psi}\), we may assume from this that \(\psi\) is positive definite on \(\kernel{\phi}\).
    Thus from continuity of quadratic forms, this means that \(\phi+\delta\psi\) is positive semidefinite for any sufficiently small \(\delta\in\bbR_{>0}\).
    Thus we know that \(\bdl\) is a second-order stationary point for \(\bar{f}\) with such sufficiently small \(\epsilon\).
    However, since \(g\notin \calN_{\sigma_k(\bdl)}\), there exists \(f\in \interior \calT_{\sigma_k(\bdl)}\) such that \(\pAngle{\sigma_k(\bdl)-\bar{f}}{f}=-c\pAngle{g}{f}<0\), which implies that \(\bdl\) is not a global minimum of~\eqref{eq:SOSMinimizationProblem}.

    Finally we assume that \(\sAngle{\bdl}\subseteq R\) is real radical.
    By Lemma~\ref{lemma:RealRadicalTangentConeCondition}, we see that the only-if condition in (ii) is always satisfied.
    It remains to show (ii)$\implies$(iii), since (iii)$\implies$(i) is trivial.
    In the proof of Lemma~\ref{lemma:RealRadicalTangentConeCondition}, we pick \(g\perp\sAngle{\bdl}_2\), \(g\notin\Sigma_X^*\) such that for any \(\bdh=(h_1,\dots,h_k)\in\syzygy_1(\bdl)\), if \(\pAngle{g}{\sigma_k(\bdh)}=0\), then \(h_1,\dots,h_k\in\sAngle{\bdl}_1\).
    Moreover, by rescaling \(g\) if necessary, we may assume that the quadratic map \(\bdh\mapsto2\pAngle{g}{\sigma_k(\bdl)}+4\nVert{\sum_{i=1}^{k}l_ih_i}^2\) is positive semidefinite with its kernel contained in \(\syzygy_1(\bdl)\).
    Therefore Lemma~\ref{lemma:LocalMinimumSufficientCondition} ensures that \(\bdl\) is a spurious local minimum for the chosen \(g\).
\end{proof}

We remark that condition (ii) in Theorem~\ref{thm:CharacterizationSpuriousMinima} does not depend on the choice of the inner product \(\pAngle{\cdot}{\cdot}\) on \(R_2\), as the statement only involves the linear functional \(\pAngle{g}{\cdot}:R_2\to\bbR\) instead of \(g\in R_2\) itself.
Thus it simplifies our discussion as we may pick certain inner product, or coordinates of \(\bbP^n\) in the constructions of examples (e.g., Example~\ref{ex:TernaryQuarticsBoundaryPoint}).
Below we mention some useful consequences of Theorem~\ref{thm:CharacterizationSpuriousMinima}.
In particular, Corollary~\ref{cor:LinearDependenceNotLocalMinimum} resembles condition (d) in~\cite[Theorem 4.1]{burer2005local}.

\begin{corollary}\label{cor:LinearDependenceNotLocalMinimum}
     If \(l_1,\dots,l_k\) are linearly dependent, then \(\bdl=(l_1,\dots,l_k)\in R_1^k\) is not a spurious second-order stationary point.
\end{corollary}
\begin{proof}
    By assumption, there exist \(c_1,\dots,c_k\in\bbR\) such that \(\sum_{i=1}^k c_i^2=1\) and \(\sum_{i=1}^k c_il_i=0\).
    Take any \(g\in R_1\) and let \(\bdh:=(h_1,\dots,h_k)=(c_1g,\dots,c_kg)\), we have \(\bdh\in\syzygy_1(\bdl)\) and \(\sigma_k(\bdh)=g^2\).
    This shows that \(\conv\sigma_k(\syzygy_1(\bdl))=\Sigma_X\) and thus \(\clos\conv\calR_\bdl=\calT_{\sigma_k(\bdl)}\).
    By Theorem~\ref{thm:CharacterizationSpuriousMinima}, we know that \(\bdl\) cannot be a spurious second-order stationary point for any \(\bar{f}\in R_2\).
\end{proof}

\begin{corollary}\label{cor:LocalMinimaOrthogonalTransform}
    Let \(O\in\bbR^{k\times k}\) be an orthogonal matrix.
    For any \(\bdl:=(l_1,\dots,l_k)\in R_1^k\), \(O\bdl\) is a spurious second-order stationary point if and only if so is \(\bdl\).
\end{corollary}
\begin{proof}
    It is easy to see that \(\sAngle{\bdl}_2=\sAngle{O\bdl}_2\).
    Take any \(\bdh\in\syzygy_1(\bdl)\).
    Then \(O\bdh\in\syzygy_1(O\bdl)\) because   $(O\bdh)^\transpose (O \bdl) = \bdh^\transpose \bdl = 0$.
    Since \(\sigma_{k}(O\bdh)=\sigma_{k}(\bdh)\) and \(\sAngle{\bdh}_2=\sAngle{O\bdh}_2\), from the second condition in Theorem~\ref{thm:CharacterizationSpuriousMinima}, we see that \(O\bdl\) is a spurious second-order stationary point if and only if so is \(\bdl\).
\end{proof}

\begin{corollary}\label{cor:NoSpuriousMinimaWithMoreSquares}
    Fix a variety \(X\subset\bbP^n\).
    Suppose \(\bdl:=(l_1,\dots,l_k)\in R_1^k\) is not a spurious second-order stationary point, then neither is \(\bdl':=(l_1,\dots,l_k,l_{k+1})\in R_1^{k+1}\) for any \(l_{k+1}\in R_1\).
\end{corollary}
\begin{proof}
    From condition (ii) in Theorem~\ref{thm:CharacterizationSpuriousMinima}, for any \(g\in\calR_\bdl^*\setminus \calT_{\sigma_k(\bdl)}^*\), there exists \(\bdh\in\syzygy_1(\bdl)\) with \(\pAngle{g}{\sigma_k(\bdh)}=0\) such that \(g\not\perp\sAngle{\bdh}_2\).
    Note that \(\sAngle{\bdl}_2\subseteq\sAngle{\bdl'}_2\).
    Moreover, for any \(\bdh=(h_1,\dots,h_k)\in\syzygy_1(\bdl)\), \(\bdh':=(h_1,\dots,h_k,0)\in\syzygy_1(\bdl')\), so \(\sigma_k(\syzygy_1(\bdl))\subseteq\sigma_{k+1}(\syzygy_1(\bdl'))\).
    Thus \(\calR_\bdl=\clos\conv(\sAngle{\bdl}_2+\sigma_k(\syzygy_1(\bdl)))\subseteq\clos\conv(\sAngle{\bdl'}_2+\sigma_k(\syzygy_1(\bdl')))=\calR_{\bdl}\), or equivalently, \(\calR_\bdl^*\supseteq\calR_{\bdl'}^*\).
    Therefore, for any \(g\in\calR_{\bdl'}^*\setminus \calT_{\sigma_{k+1}(\bdl')}^*\), the syzygy \(\bdh\in\syzygy_1(\bdl)\) with \(\pAngle{g}{\sigma_k(\bdh)}=0\) such that \(g\not\perp\sAngle{\bdh}_2\), can be extended to \(\bdh':=(h_1,\dots,h_k,0)\in\syzygy_1(\bdl')\) satisfying \(\pAngle{g}{\sigma_{k+1}(\bdh')}=0\) and \(g\not\perp\sAngle{\bdh'}_2=\sAngle{\bdh}_2\).
    This shows that \(\bdl'\) is not a spurious second-order stationary point.
\end{proof}

\section{Spurious local minima on varieties of minimal degree}

In this section, we focus on varieties of minimal degree and prove Theorems~\ref{thm:ExistenceSpuriousMinima}, \ref{thm:SurfacesOfMinimalDegree}, and \ref{thm:VeroneseSurfaceMinima}.
A useful property regarding varieties of minimal degree is that they are arithmetically Cohen-Macaulay, which implies that for \(k=\dim(X)+1\), whenever the linear forms \(l_1,\dots,l_k\) do not have a common zero on \(X\), then \(\sAngle{l_1,\dots,l_k}=R_2\)~\cite[Lemma 2.2]{blekherman2019low}.
In particular, this implies that the Jacobian matrix \(\diff_\bdl\sigma_k\) has full rank for any tuple of linear forms \(\bdl=(l_1,\dots,l_k)\) that share no common zero on \(X\), and thus \(\bdl\) cannot be a (first-order) stationary point for any choice of \(\bar{f}\).
We illustrate the benefit of studying the syzygies by the following lemma for rational normal curves, which reproduces and slightly simplifies the argument used in~\cite{legat2023low}.
\begin{lemma}\label{lemma:RationalNormalCurves}
    Let \(X=\nu_n(\bbP^1)\subset\bbP^n\) be a rational normal curve.
    Then there is no spurious second-order stationary point in Problem~\eqref{eq:SOSMinimizationProblem}.
\end{lemma}
\begin{proof}
    By Corollary~\ref{cor:NoSpuriousMinimaWithMoreSquares}, it suffices to consider the case \(k=\Pythagoras(X)=\dim(X)+1=2\).
    We may further assume that \(V(l_1,l_2)\neq\varnothing\) as otherwise we know that \(\bdl\) is not a first-order stationary point.
    Let \(a_i:=\nu_n^\sharp(l_i)\in\bbR[x_0,x_1]_n\) denote the pullback images of \(l_i\), which is a degree-\(n\) forms on \(\bbP^1\) for each \(i=1,2\), and \(a\) their greatest common divisor, i.e., \(a_1=ab_1\) and \(a_2=ab_2\) for some \(b_1,b_2\in\bbR[x_0,x_1]\).
    Then the degree-\(n\) syzygies of \(a_1,a_2\) are of the form \((b_2c,-b_1c)\) for any \(c\in\bbR[x_0,x_1]_{\deg(a)}\) as \(\bbR[x_0,x_1]\) is an integral domain.
    Consequently, the sum of squares of these syzygies can be written as \((b_1^2+b_2^2)c^2\).

    Applying a change of coordinates on \(\bbP^1\) if needed, we may assume that the finite set of zeros \(V(a_1,a_2)=\nu_n^{-1}(V(l_1,l_2))\subset\{x_0\neq0\}\subset\bbP^1\) is affine.
    Let \(\tilde{a}_i(x):=a(1,x),\tilde{b}_i(x):=b(1,x)\) for \(i=1,2\), and \(\tilde{a}(x):=a(1,x),\tilde{c}(x):=c(1,x)\) be the corresponding univariate polynomials.
    By Theorem~\ref{thm:CharacterizationSpuriousMinima}, we want to show that for any \(g\in\sAngle{l_1,l_2}^\perp\setminus\Sigma_X^*\), the quadratic form \((h_1,h_2)\mapsto\pAngle{g}{h_1^2+h_2^2}\) is not positive semidefinite on the subspace \(\syzygy_1(l_1,l_2)\).
    Under the pullback map \(\nu_n^\sharp:R\to\bbR[x_0,x_1]\) and the dehomogenization \(\bbR[x_0,x_1]\to\bbR[x_0,x_1]/(x_0-1)\cong\bbR[x]\), this is equivalent to say that any linear functional \(L:\bbR[x]_{\le2n}\to\bbR\) satisfying
    \begin{enumerate}
        \item \(L(\tilde{a}_1u_1+\tilde{a}_2u_2)=0\) for any polynomials \(u_1,u_2\in\bbR[x]_{\le n}\), and
        \item \(L(\sum_{i=1}^{m}v_i^2)<0\) for some \(v_1,\dots,v_m\in\bbR[x]_{\le n}\) and some \(m\ge1\),
    \end{enumerate}
    must give \(L((\tilde{b}_1^2+\tilde{b}_2^2)\tilde{c}^2)<0\) for some \(\tilde{c}\in\bbR[x]_{\le\deg(a)}\).
    To express \(L\) as point and derivative evaluations of the polynomials with degree no more than \(2n\), suppose \(V(\tilde{a})=\{y_1,\dots,y_p,z_1,\bar{z}_1,\dots,z_q,\bar{z}_q\}\subset\bbC\) and
    \begin{equation*}
        \tilde{a}(x)=\prod_{r=1}^{p}(x-y_r)^{\mu_r}\prod_{s=1}^{q}(x-z_s)^{\lambda_s}(x-\bar{z}_s)^{\lambda_s},
    \end{equation*}
    where \(y_1,\dots,y_{p}\in\bbR\) are real roots of \(\tilde{a}\), with multiplicities \(\mu_1,\dots,\mu_{p}\), and \(z_1,\bar{z}_1,\dots,z_{q},\bar{z}_{q}\in\bbC\setminus\bbR\) are pairs of complex roots, with multiplicities \(\lambda_1,\dots,\lambda_{q}\). 
    Then from property (i) of \(L\), there exist real coefficients \(\alpha_r^{(i)},\beta_{s}^{(i)},\gamma_{s}^{(i)}\) such that the image of any \(g\in\bbR[x]_{\le2n}\) under \(L\) can be written as
    \begin{equation*}
        \begin{aligned}
            L(g)&=\sum_{r=1}^{p}\left[\alpha_r^{(1)}g(y_r)+\sum_{i=2}^{\mu_r}\alpha_r^{(i)}\frac{\diff^i}{\diff y^i}g\vert_{x=y_r}\right]\\
                &+\sum_{s=1}^{q}\left[\beta_s^{(1)}\operatorname{re}(g(z_s))+\gamma_s^{(1)}\operatorname{im}(g(z_s))+\sum_{i=2}^{\mu_s}\beta_s^{(i)}\operatorname{re}\left(\frac{\diff^i}{\diff x^i}g\vert_{x=z_s}\right)+\gamma_s^{(i)}\operatorname{im}\left(\frac{\diff^i}{\diff x^i}g\vert_{x=z_s}\right)\right].
        \end{aligned}
    \end{equation*}
    Note that if \(\mu_r>1\) with \(\alpha_{r}^{\mu_r}\neq0\) for some \(r=1,\dots,p\), then there is a unique term in \(L((\tilde{b}_1^2+\tilde{b}_2^2)\tilde{c}^2)\)
    \begin{equation*} 
    \alpha_{r}^{(\mu_r)}\cdot2(\tilde{b}_1^2+\tilde{b}_2^2)\tilde{c}\cdot\frac{\diff^{\mu_r}}{\diff x^{\mu_r}}\tilde{c}\vert_{x=y_r}.
    \end{equation*}
    Using Hermite interpolation, we can always find \(\tilde{c}\in\bbR[x]_{\le\deg(a)}\) such that this term is sufficiently negative.
    Consequently, we can make \(L((\tilde{b}_1^2+\tilde{b}_2^2)\tilde{c}^2)\) negative, by the choice of \(\tilde{c}\).
    For example, an explicit way to do this is to set \(\tilde{c}(y_r)=1\), \(\frac{\diff^i}{\diff x^i}\tilde{c}\vert_{x=y_r}=0\) for \(i=2,\dots,\mu_r-1\), and \(\frac{\diff^i}{\diff x^i}\tilde{c}\vert_{x=y_r}=-C\) for any \(C>0\), while we let the evaluation of \(\tilde{c}\) vanish at all other roots \(x\in V(\tilde{a})\setminus\{y_r\}\) alongside with their derivatives.
    The same argument works for the case when \(\lambda_s>1\) with \(\beta_s^{(\lambda_s)}\) or \(\gamma_s^{(\lambda_s)}\neq0\) by noting that for any \(\tilde{b}_1(z_s),\tilde{b}_2(z_s),\) and \(\tilde{c}(z_s)\), there exists \(C\in\bbC\) such that
    \[
        \beta_s^{(\lambda_s)}\operatorname{re}\left(2(\tilde{b}_1^2+\tilde{b}_2^2)(z_s)\tilde{c}(z_s)\cdot C\right)
        +\gamma_s^{(\lambda_s)}\operatorname{im}\left(2(\tilde{b}_1^2+\tilde{b}_2^2)(z_s)\tilde{c}(z_s)\cdot C\right)<0.
    \]

    It remains to examine the case where \(\mu_1,\dots,\mu_p,\lambda_1,\dots,\lambda_q\) are all 1 with nonzero coefficients in \(L\).
    If \(\beta_s^{(1)}\) and \(\gamma_s^{(1)}\) are not simultaneously 0 for some \(s=1,\dots,q\), then it is well-known that \(L\) defines an indefinite quadratic form, by considering polynomials \(\tilde{c}\) that vanish at \(V(\tilde{a})\setminus\{z_s,\bar{z}_s\}\).
    However, if \(\beta_s^{(1)}=\gamma_s^{(1)}=0\) for all \(s=1,\dots,q\), then from property (ii) of \(L\), we must have \(\alpha_r^{(1)}<0\) for some \(r=1,\dots,p\), in which case we can make \(L((\tilde{b}_1^2+\tilde{b}_2^2)\tilde{c}^2)\) negative by considering polynomials \(\tilde{c}\) that vanish at \(V(\tilde{a})\setminus\{y_r\}\).
\end{proof}

To answer the question on the existence of spurious second-order stationary points and local minima beyond rational normal curves, we consider the following examples.
The first example shows the existence of second-order stationary points that are not local minima on the Veronese surface.

\begin{example}\label{ex:TernaryQuarticsBoundaryPoint}
    Consider the Veronese surface \(X=\nu_2(\bbP^2)\subset\bbP^5\) where the map $\nu_2\colon \bbP^2 \to \bbP^5$ is given by \(\nu_2([x_0\mathcolon x_1\mathcolon x_2])=[x_0^2\mathcolon x_0x_1\mathcolon x_0x_2\mathcolon x_1^2\mathcolon x_1x_2\mathcolon x_2^2]\). 
    By slight abuse of notation we identify quadratic (resp.\ quartic) forms on \(\bbP^2\) with linear (resp.\ quadratic) forms on \(X\).
    Fix any inner product on \(R_2\) such that the forms corresponding to all the monomials \(x_0^4,x_0^3x_1,\dots,x_2^4\) are pairwise orthogonal, and take \(k=\Pythagoras(X)=3\).
    We claim that the tuple \(\bdl=(x_0^2,x_0x_1,x_1^2)\) is a spurious second-order stationary point.
    To see this, we first note that \(\sAngle{\bdl}_2\) contain all monomials whose degree in \(x_2\) is less than 3, because the sum of degrees in \(x_0\) and \(x_1\) is at least 2.
    In other words, \(\sAngle{\bdl}_2\) is the orthogonal complement to the subspace \(\linspan_\bbR\{x_2^4,x_0x_2^3,x_1x_2^3\}\).
    Thus to make it a first-order stationary point, we can take \(\bar{f}=\sigma_3(\bdl)+g\), where \(g\perp\sAngle{\bdl}_2\).
    We claim that \(g=-\epsilon x_2^4\) for some \(\epsilon>0\) would further satisfy the second-order optimality condition.
    To see this, note that for any \(\bdh=(h_1,h_2,h_3)\in\syzygy_1(\bdl)\), we can write \(h_i=a_i x_2^2+h'_i\), for some \(a_i\in\bbR\) such that \(h'_i\in H:=\linspan_\bbR\{x_0^2,x_0x_1,x_0x_2,x_1^2,x_1x_2\}\).
    By definition
    \[
        x_0^2 h_1 + x_0x_1 h_2 + x_1^2 h_3=a_1x_0^2x_2^2+a_2x_0x_1x_2^2+a_3x_1^2x_2^2+h'=0,
    \]
    where \(h'\in\bbR[x_0,x_1,x_2]_4\) does not contain any monomial that is divisible by \(x_2^2\).
    Thus we must have \(a_1=a_2=a_3=0\) by the linear independence of the monomials \(x_0^2x_2^2,x_0x_1x_2^2,x_1^2x_2^2\), and consequently \(h_i\in H\) for \(i=1,2,3\).
    Therefore, \(\sum_{i=1}^{3}h_i^2\in\sAngle{\bdl}_2\perp g\) for any \(g\) satisfying the first-order optimality condition. 
    By Theorem~\ref{thm:CharacterizationSpuriousMinima}, we know that \(\bdl\) is a spurious second-order stationary point for any sufficiently small \(\epsilon>0\) if \(g\notin\Sigma_X^*\) and in addition \(g\perp\sAngle{\bdh}_2\) for any \(\bdh\in\syzygy_1(\bdl)\).
    This means that \(g\perp\sAngle{H}_2\), which holds exactly when \(g=-\epsilon x_2^4\).

    In the following we show that \(\bdl\) is not a spurious local minimum for such choice of \(g\).
    Consider a curve \(\gamma:(-1,1)\to R_1^k\), \(\gamma(z)=\bdl+\bdl^{(1)}z+\bdl^{(2)}z^2\) with \(\bdl^{(1)}=\sqrt{2}(x_1x_2,-x_0x_2,0)\) and \(\bdl^{(2)}=(-x_2^2,0,-x_2^2)\).
    Clearly, \(\bdl^{(1)}\in\syzygy_1(\bdl)\).
    Moreover, 
    \begin{align*}
        \sum_{i=1}^{3}2l_il_i^{(2)}=-\sum_{i=1}^{3}[l_i^{(1)}]^2=-2x_2^2(x_0^2+x_1^2),\quad 
        \sum_{i=1}^{3}l_i^{(1)}l_i^{(2)}=\sqrt{2}x_2^3(x_0-x_1)
    \end{align*}
    are all perpendicular to \(g=-\epsilon x_2^4\).
    Thus the objective function along this curve can be written as 
    \begin{equation*}
        \phi(z)=z^4\Bigl(2\pAngle{g}{\sigma_3(\bdl^{(2)})}+\nVert{\sum_{i=1}^{3}2l_il_i^{(2)}+[l_i^{(1)}]^2}^2\Bigr)+o(z^4)=-4\nVert{x_2^4}^2\epsilon z^4+o(z^4),
    \end{equation*}
    which implies that it is negative for some sufficiently small \(z>0\).
    Consequently, \(\bdl\) cannot be a local minimum.

    Before ending this example, we want to make two remarks.
    First, to be a second-order stationary point, \(g\) cannot contain other monomials in \(\linspan_\bbR\{x_2^4,x_0x_2^3,x_1x_2^3\}\).
    To see this, suppose \(g=-\epsilon x_2^4+\delta_1 x_0x_2^3+\delta_2 x_1x_2^3\) for some \(\delta_1,\delta_2\in\bbR\).
    Then for \(\bdh=(x_1x_2,-x_0x_2,0)\in\syzygy_1(\bdl)\), we see that \(\pAngle{g}{\sigma_3(\bdh)}=\pAngle{-\epsilon x_2^4+\delta_1 x_0x_2^3+\delta_2 x_1x_2^3}{x_1^2x_2^2+x_0^2x_2^2}=0\).
    However, both \(x_0x_2^3,x_1x_2^3\in\sAngle{\bdh}_2\), so \(g\perp\sAngle{\bdh}_2\) only if \(\delta_1=\delta_2=0\).
    This implies that the choice of \(g\) for second-order optimality condition is unique up to scaling of \(\epsilon>0\).

    Second, the linear forms \(l_1,l_2,l_3\) share a common real zero \(\nu_2([0\mathcolon0\mathcolon1])\in\bbP^5\), and thus the quadratic form \(\sigma_3(\bdl)\) lies on the boundary of \(\Sigma_X\).
    However, the ideal \(\sAngle{\bdl}\) is not real radical, since \(x_0^2x_2^2\in\sAngle{\bdl}\) but \(x_0x_2\notin\sAngle{\bdl}\).
    Moreover, any tuple \(\bdl_b:=(x_0^2+bx_1^2,\sqrt{1-2b}x_0x_1,\sqrt{1-b^2}x_1^2)\) satisfies \(\sigma_3(\bdl_b)=x_0^4+x_0^2x_1^2+x_1^4=\sigma_3(\bdl)\), for \(0\le b\le\frac{1}{2}\), which means that \(\sigma_3(\bdl)\) has infinitely many representations as a sum of 3 squares.
    We will discuss more on the number of representations of the sums of squares on the Veronese surface in Proposition~\ref{prop:VeroneseSurfaceMinimaBinary}.
\end{example}

The second example shows the existence of spurious local minima on a 2-dimensional rational normal scroll.
As rational normal scrolls are toric varieties, they can be described through their associated lattice polytopes, which are known as \emph{Lawrence prisms} defined by~\cite{batyrev2006multiples}
\begin{equation}\label{eq:LawrencePrism}
    P:=\conv\{e_0,e_0+n_1(e_m-e_0),e_1,e_1+n_2(e_m-e_0),\dots,e_{m-1},e_{m-1}+n_m(e_m-e_0)\},
\end{equation}
where \(e_0,e_1,\dots,e_m\in\bbR^{m+1}\) is the standard basis, and \(n_1,\dots,n_m\in\bbZ_{\ge1}\) are called \emph{heights} of the prism.
For instance, a Lawrence prism of heights \((1,2,2)\) is shown in Figure~\ref{fig:LawrencePrism}, where each lattice point corresponds to a bihomogeneous monomial on \(\bbP^2\times\bbP^1\).

\begin{figure}[h]
    \centering
    \begin{tikzpicture}[scale=1, x={(2.5cm,-0.6cm)}, y={(1.5cm,0.8cm)}, z={(0cm,1cm)}]
    % Define extreme points
    \coordinate (000) at (0,0,0);
    \coordinate (010) at (0,1,0);
    \coordinate (100) at (1,0,0);
    \coordinate (001) at (0,0,1);
    \coordinate (101) at (1,0,1);
    \coordinate (102) at (1,0,2);
    \coordinate (011) at (0,1,1);
    \coordinate (012) at (0,1,2);
    % Draw edges of the prism
    \draw[dotted] (000) -- (010) -- (100) -- cycle;
    \draw (000) -- (100);
    \draw (000) -- (001);
    \draw[dotted] (010) -- (012);
    \draw (100) -- (102);
    \draw (001) -- (012) -- (102) -- cycle;
    % Draw vertices
    \fill (000) circle (2pt) node [left=2mm] {$x_0y_0$};
    \fill (001) circle (2pt) node [left=1mm] {$x_0y_1$};
    \fill (010) circle (2pt) node [left=2mm] {$x_1y_0^2$};
    \fill (011) circle (2pt) node [left=1mm] {$x_1y_0y_1$};
    \fill (012) circle (2pt) node [left=2mm] {$x_1y_1^2$};
    \fill (100) circle (2pt) node [right=2mm] {$x_2y_0^2$};
    \fill (101) circle (2pt) node [right=1mm] {$x_2y_0y_1$};
    \fill (102) circle (2pt) node [right=2mm] {$x_2y_1^2$};
    \end{tikzpicture}
    \caption{A Lawrence Prism of Heights \((1,2,2)\)}
    \label{fig:LawrencePrism}
\end{figure}
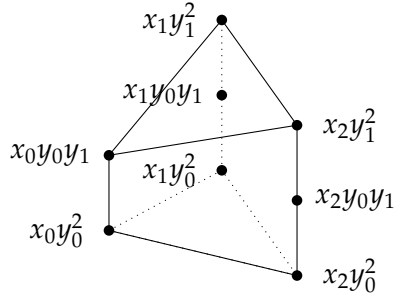

\begin{example}\label{ex:2DimScrollSpuriousMinima}
    Let \(X\) be a 2-dimensional rational normal scroll associated with the Lawrence prism of heights \((2,2)\) and take \(k=\Pythagoras(X)=3\).
    Through the toric parametrization of \(X\), we can use monomials as a basis of $R_1$ and $R_2$, namely \( (y_0^2x_1,y_0y_1x_1,y_1^2x_1,y_0^2x_2,y_0y_1x_2,y_1^2x_2)\) for $R_1$ and \(y_0^4x_1^2,\dots,y_1^4x_2^2\) for \(R_2\). As inner product on $R_2$, we choose the one that makes these monomials an orthonormal basis (it is sufficient to choose one such that they are orthogonal).
    Set \(\bdl=(y_0^2x_1,y_0y_1x_1,y_1^2x_1)\). 
    Any syzygies with all of their components in \(\sAngle{\bdl}_1\) do not affect the second-order optimality condition in Theorem~\ref{thm:CharacterizationSpuriousMinima}, so we first calculate the syzygies modulo \(y_0^2x_1,y_0y_1x_1\), and \(y_1^2x_1\).
    Through straightforward calculation on the dimension, they are generated by the tuples \(\bdh^{(1)}:=(y_0y_1x_2,-y_0^2x_2,0)\), \(\bdh^{(2)}:=(y_1^2x_2,-y_0y_1x_2,0)\), \(\bdh^{(3)}:=(0,-y_0y_1x_2,y_0^2x_2)\), and \(\bdh^{(4)}:=(0,-y_1^2x_2,y_0y_1x_2)\).
    Thus for any \(\bdh=(h_1,h_2,h_3)\in\syzygy_1(\bdl)\), \( \bdh = c_1\bdh^{(1)} + c_2\bdh^{(2)} + c_3\bdh^{(3)} + c_4\bdh^{(4)}\) with \(c_1,c_2,c_3,c_4\in\bbR\), there exists some \(f\in\sAngle{\bdl}_2\) such that
    \begin{align*}
        \sum_{i=1}^{3}h_i^2=&f+(c_1 y_0y_1x_2+c_2 y_1^2x_2)^2+(c_1 y_0^2x_2+(c_2+c_3)y_0y_1x_2+c_4 y_1^2x_2)^2+(c_3 y_0^2x_2+c_4 y_0y_1x_2)^2\\
        =&f+(c_1^2+c_3^2)y_0^4x_2^2+(2c_1c_2+2c_1c_3+2c_3c_4)y_0^3y_1x_2^2+(c_1^2+c_2^2+2c_2c_3+c_3^2+2c_1c_4+c_4^2)y_0^2y_1^2x_2^2\\
         &\;+(2c_1c_2+2c_2c_4+2c_3c_4)y_0y_1^3x_2^2+(c_2^2+c_4^2)y_1^4x_2^2.
     \end{align*}
     We can choose \(g=x_2^2(y_0^4+y_1^4-\frac{1}{3}y_0^2y_1^2)\), which clearly satisfies \(g\perp\sAngle{\bdl}_2\), \(g\notin\Sigma_X^*\) (indeed, \(\pAngle{g}{y_0^2y_1^2x_2^2}=-\frac{1}{3}<0\)), and \(\pAngle{g}{\sum_{i=1}^{3}h_i^2}=\frac{1}{3}[c_1^2+c_4^2+(c_1-c_4)^2+c_2^2+c_3^2+(c_2-c_3)^2]\), which is zero only if \(c_1=\dots=c_4=0\), and thus \(h_i\in\sAngle{\bdl}_1\) for \(i=1,2,3\) so \(g\perp\sAngle{\bdh}_2\).
     This implies that \(\bdl\) is a spurious second-order stationary point, and moreover a spurious local minimum by Theorem~\ref{thm:CharacterizationSpuriousMinima}, since the ideal \(\sAngle{\bdl}\subset R\) is real radical.
\end{example}

The construction in Example~\ref{ex:2DimScrollSpuriousMinima} can be generalized to any higher-dimensional smooth rational normal scroll, which, together with Example~\ref{ex:TernaryQuarticsBoundaryPoint}, shows the only-if part in Theorem~\ref{thm:ExistenceSpuriousMinima} due to the classification of smooth varieties of minimal degree~\cite{eisenbud1987varieties}.

\begin{proposition}\label{prop:ScrollSpuriousMinima}
    Let \(X\subset\bbP^n\) be any \(m\)-dimensional smooth rational normal scroll with a Lawrence prism of heights \(0<n_1\le n_2\le\cdots\le n_m\) for some \(m\ge2\).
    Then there exists a spurious local minimum in the nonconvex formulation~\eqref{eq:SOSMinimizationProblem} of rank \(k=n-n_1\). 
    Consequently, there are spurious second-order stationary points for any \(m+1\le k\le n-n_1\).
\end{proposition}
\begin{proof}
    Similar to Example~\ref{ex:2DimScrollSpuriousMinima}, we choose a monomial basis of \(R_1\) consisting of \(y_0^iy_1^{n_j-i}x_j\) where $j$ ranges from \(1,\dots,m\) and \(i=0,1,\dots,n_j\).
    We show the existence of a spurious local minimum for \(k=n-n_1\).
    Let \(\bdl\) be the tuple consisting of \(l_{i,j}:=y_0^iy_1^{n_j-i}x_j\) for \(i=0,\dots,n_j\) and \(j=2,\dots,m\).
    From the determinantal representation of \(I_X\), we know that the nontrivial syzygies of \(l_{i,j}\) are generated by tuples \(g_{d,i,j}\in R_1^k\) with nonzero entries \(y_0^dy_1^{n_1-d}x_1\) for \(l_{i,j}\) and \(-y_0^{d+1}y_1^{n_1-d-1}x_1\) for \(l_{i-1,j}\), for any \(d=0,1,\dots,n_1-1\), \(i=1,\dots,n_j\), and \(j=2,\dots,m\).
    Thus the sums of squares of the tuple of syzygies \(\sum_{i,j,d}c_{d,i,j}g_{d,i,j}\) can be written as
    \[
        \sum_{i,j}(\sum_{d=0}^{n_1}c_{d,i,j}y_0^dy_1^{n_1-d}x_1-c_{d-1,i+1,j}y_0^{d-1}y_1^{n_1-d+1}x_1)^2,
    \]
    where the coefficients \(c_{d,i,j}\in\bbR\) with the convention \(c_{d,i,j}=0\) when \(d<0\) or \(d>n_1\), for each \(i=0,1,\dots,n_j\) and \(j=2,\dots,m\).
    Thus \(y_0^{2n_1}x_1^2\) cannot be represented in the above form because the coefficient of \(y_1^{2n_1}x_1^2\) is \(\sum_{i,j}c_{0,i,j}^2=0\), which implies that each \(c_{0,i,j}=0\) and thus the coefficient of \(y_0^2y_1^{2n_1-2}x_1^2\) is \(\sum_{i,j}c_{1,i,j}^2=0\), so on and so forth.
    Note that \(R/\sAngle{\bdl}\) is isomorphic to the coordinate ring of degree-\(n_1\) rational normal curve \(X'\), so \(\sAngle{\bdl}\) is real radical, and \(y_0^{2n_1}x_1^2\) corresponds to an extreme ray of \(\Sigma_{X'}\).
    Therefore, we conclude that there exists \(f'\notin\Sigma_{X'}^*\) that is positive semidefinite on the nontrivial syzygies, which can then be extended to \(f\notin\Sigma_X^*\) satisfying \(f\perp\sAngle{\bdl}_2\) and \(\pAngle{f}{\sigma_k(\bdh)}\ge0\) for any \(\bdh\in\syzygy_1(\bdl)\).
    Lemma~\ref{lemma:RealRadicalTangentConeCondition} shows that when \(\sAngle{\bdl}\) is real radical, \(\pAngle{f}{\sigma_k(\bdl)}=0\) implies that \(f\perp\sAngle{\bdh}_2\), proving that \(\bdl\) is a spurious local minimum by Theorem~\ref{thm:CharacterizationSpuriousMinima}.
    The existence of spurious second-order stationary points for any \(m+1\le k\le n-n_1\) then follows from Corollary~\ref{cor:NoSpuriousMinimaWithMoreSquares}.
\end{proof}

Proposition~\ref{prop:ScrollSpuriousMinima} indicates that even for varieties of minimal degree, spurious local minima may persist with up to \(k=n-1\) squares (if \(n_1=2\)), which is larger than the known ranks associated with spurious local minima for Burer-Monteiro-type algorithms~\cite{ocarroll2022burer}.

\subsection{Proof of Theorem~\ref{thm:ExistenceSpuriousMinima}}

\begin{proof}[Proof of Theorem~\ref{thm:ExistenceSpuriousMinima}]
    The only-if direction is shown by Example~\ref{ex:TernaryQuarticsBoundaryPoint} and Proposition~\ref{prop:ScrollSpuriousMinima}, so we discuss the if direction below. 
    \begin{itemize}
        \item When \(\dim(X)=1\), then \(X\) is a rational normal curve and the assertion is shown in Lemma~\ref{lemma:RationalNormalCurves}, which is also the main result in~\cite{legat2023low}. 
        \item When \(\dim(X)=n\), then \(X=\bbP^n\) and either \(\sAngle{\bdl}_2=R_2\) or \(l_1,\dots,l_{n+1}\) are linearly dependent, which cannot be a spurious second-order stationary point by Corollary~\ref{cor:LinearDependenceNotLocalMinimum}.
        \item When \(\dim(X)=n-1\), then \(X\subseteq\bbP^n\) is a smooth quadratic hypersurface, which we assume to be defined by a single quadric \(Q\in S_2\).
        We can identify $R_1$ with $S_1$ as real vector spaces, and write the image of $L_1,\dots,L_n\in S_1$ under this isomorphism as $l_1,\dots,l_n$.
        Assume $\bdl=(l_1,\dots,l_n)\in R_1^n$ is a second-order stationary point. By Theorem~\ref{thm:CharacterizationSpuriousMinima} there exists nonzero $g\in R_2$ which satisfies $g\in\sAngle{\bdl}_2^\perp\setminus\Sigma_X^*$, and satisfies $\pAngle{g}{\sigma_k(\bdh)}\geq 0$ for every $\bdh=(h_1,\dots, h_n)\in\syzygy_1(\bdl)$.

        By Corollary~\ref{cor:LinearDependenceNotLocalMinimum} the forms $l_1,\dots,l_n$ and thus $L_1,\dots,L_n$ are linearly independent so they can be completed to a basis $l_0,l_1,\dots,l_n$ of $R_1$ and correspondingly $L_0,L_1,\dots,L_n$ of $S_1$. Since $g$ annihilates the quadratic part of the ideal $\sAngle{\bdl}_2\subset R_2$ we have $g=a l_0^2$ for some real number $a$. 
        If $Q$ is the quadric defining $X$ then $Q\in \sAngle{L_1,\dots,L_n}\subset S_2$ since otherwise $\sAngle{l_1,\dots,l_n}=R_2$, proving that the differential of $\sigma_n$ is of full rank at $\bdl$ contradicting the fact that $\bdl$ is a first-order stationary point. The form $Q$ must furthermore be of full rank because $X$ is nonsingular and therefore must involve $l_0$. It follows that if the vector $\bdh=(h_1,\dots,h_n)\in S_1^n$ satisfies $\sum_{i=1}^{n}l_ih_i=Q$ and $h_i:=c_{i0}l_0+\sum_{j=1}^{n}c_{ij}l_j$ for some $c_{ij}\in\bbR$, then some of the coefficients $c_{i0}$ must be nonzero. Since $\pAngle{g}{\sigma_k(\bdh)}= a\sum_{i=1}^{n}c_{i0}^2\ge0$, we conclude that $a\geq 0$, and thus $g\in-\Sigma_X^*$, a contradiction with the pointedness of $\Sigma_X^*$.
            \qedhere
    \end{itemize}
\end{proof}

\subsection{Proof of Theorem~\ref{thm:SurfacesOfMinimalDegree}}
\label{sec:SurfacesOfMinimalDegree}

In view of Corollary~\ref{cor:LocalMinimaOrthogonalTransform}, it is possible to consider equivalence classes of tuples of linear forms.
We say \(F\in S_2\) is a Gram matrix representing \(f\in R_2\) if \(F\) is positive semidefinite and \(f\) is the image of \(F\) under the quotient map \(\kappa:S_2\to R_2\).
We denote all Gram matrices with rank at most \(k\) as \(\calG_k\subseteq S_2\).
Each equivalence class of tuples of linear forms \(\bdl=(l_1,\dots,l_k)\in R_1^k\) under orthogonal transformations is associated with a unique Gram matrix \(G=\sum_{i=1}^{k}L_iL_i^\transpose\in\calG_k\), where \(L_i\) is the representative of \(l_i\) in \(S_1\).
For simplicity, we use \(\tau:R_1^k\to\calG_k\) to denote the map that sends \(\bdl\) to \(\tau(\bdl)=\sum_{i=1}^{k}L_iL_i^\transpose\), which lets us to factor the sum-of-square map \(\sigma_k=\kappa\circ\tau\).
One can check that \(\tau\) is an open map as it is a quotient map by the \(k\times k\) orthogonal automorphism group on \(R_1^k\).

Let \(D\subset R_1^k\) denote the locus of tuples of linear forms \(\bdl=(l_1,\dots,l_k)\in R_1^k\) such that \(l_1,\dots,l_k\) do not have common zeros on \(X\).
We have seen that any tuple \(\bdl\in D\) cannot be spurious first-order stationary point as \(\sAngle{\bdl}=R_2\).
We may extend this observation: the preimage of any \(F\supseteq\tau(D)\) such that the restriction \(\kappa\vert_{F}\) is strongly open cannot contain spurious local minima because any open neighborhood of \(\bdl\in \tau^{-1}(F)\) is mapped to an open set containing a point closer to the target \(\bar{f}\) than \(\sigma_k(\bdl)\).
Our next lemma gives an explicit example of such set \(F\subset\calG_k\).

\begin{lemma}\label{lemma:FiniteGramNotSpuriousMinimum}
    Let \(X\) be a smooth variety of minimal degree, \(k=\dim(X)+1\), and \(F\subset\calG_k\) be the subset of Gram matrices such that each point in the image of the quotient map \(\kappa:\calG_k\to R_2\) has only finitely many preimages.
    Then \(F\supseteq\tau(D)\) and the restricted quotient map \(\kappa\vert_F\) is strongly open.
    In particular, any tuple \(\bdl\in\tau^{-1}(F)\) cannot be a spurious local minimum.
\end{lemma}

\begin{proof}
    For any \(\bdl\in D\), the Jacobian matrix \(\diff_\bdl\sigma_k\) has full rank, so \(\sigma_k\) is locally surjective at \(\bdl\).
    The factorization \(\sigma_k=\kappa\circ\tau\) implies that \(\kappa\) is also locally surjective at \(\tau(\bdl)\), the Gram matrix associated with \(\bdl\). 
    As \(X\) is a variety of minimal degree and thus has 0 quadratic deficiency~\cite[Section 3]{blekherman2016sums}, the dimension of the quadratic forms \(\dim_\bbR{R_2}=k(n+1)-\binom{k}{2}\), which is equal to the dimension of \(\calG_k\setminus\calG_{k-1}\) as a quotient manifold of \(R_1^k\) under the \(k\times k\) orthogonal group.
    Therefore, \(\tau\) induces a local diffeomorphism between \(\calG_k\setminus\calG_{k-1}\) and \(R_2\) at \(\bdl\), from which we see that \(\tau(\bdl)\) must be an isolated point in the preimage set \(\kappa^{-1}(\sigma_k(\bdl))\).
    The preimage set \(\kappa^{-1}(\sigma_k(\bdl))\) must then be finite because \(\calG_k\) and \(\tau\) are algebraic (i.e., defined by polynomials in finitely many variables), so \(\tau(\bdl)\in F\).
    Thus we have shown \(\tau(D)\subseteq F\).

    Now let \(\Delta\subset R_2\) denote the Zariski closure of all quadratic forms that are singular at some point on \(X\).
    Note that any common zero of the tuple of linear forms \(\bdh=(h_1,\dots,h_k)\in R_1^k\) is a singular point of \(V(\sigma_k(\bdh))\subset X\).
    This implies that \(R_1^k\setminus\sigma_k^{-1}(\Delta)\subseteq D\), and thus \(\calG_k\setminus\kappa^{-1}(\Delta)\subseteq\tau(D)\).
    Since \(\kappa\) is locally surjective at every point of \(\tau(D)\), \(\kappa\vert_{F\setminus\kappa^{-1}(\Delta)}\) is strongly open.
    We also have \(\codim{(\Delta\cap\interior{\Sigma_X})}\ge2\) in \(R_2\) as a semialgebraic set because \(\Delta\cap\interior{\Sigma_X}\) is contained in the singular locus of the discriminant of the second Veronese re-embedding of \(X\).
    Thus \(\Delta\cap\Sigma_X\) is a closed nondense subset of \(\Sigma_X\) that separates no region, which shows that \(\kappa\vert_F\) is strongly open by the properties of light mappings~\cite[Chapter VII, Theorem 2.3]{whyburn2015topological}.
    The last assertion then follows from the definition of strongly openness, and the fact that any open neighborhood of \(\sigma_k(\bdl)\) not minimizing \(\nVert{\cdot-\bar{f}}\) must contain a point that is closer to \(\bar{f}\).
\end{proof}

Lemma~\ref{lemma:FiniteGramNotSpuriousMinimum} leads to a similar result to Lemma~\ref{lemma:RationalNormalCurves} (also the main theorem in~\cite{legat2023low}) by the following observation:
any nonnegative binary form has only finitely many inequivalent representations as a sum of two squares, which correspond to different combinations of the linear forms in its linear factorization over \(\bbC\).
Thus by the Veronese embedding, we know that on a rational normal curve \(X\), each form in \(\Sigma_X\) has finitely many Gram matrices in \(\calG_2\).

We proceed to examine more closely the quadratic forms in \(\Sigma_X\) with infinitely many Gram matrices in \(\calG_k\).
A useful technique is projection away from common zeros of \(\bdl\) on \(X\).
To be more precise, let \(E\subset\bbP^n\) be a \((d-1)\)-dimensional real linear subspace spanned by $d$ points on $X$ (that can be real or come in complex pairs of points on \(X\)).
The projection away from \(E\) defines a rational map \(\pi_E:\bbP^n\dashrightarrow\bbP^{n-d}\).
We define \(X'\subseteq\bbP^{n-d}\) to be the Zariski closure of the image of \(X\setminus E\) under \(\pi_E\), which induces an inclusion \(\pi_E^\sharp:S'\to S\) where \(\bbP^{n-d}=\mathrm{Proj}(S')\).
It is well-known that \(X'\) is again a variety of minimal degree~\cite{eisenbud1987varieties} and is totally real because \(\pi_E\) sends real points to real points.
Thus the ideal \(I_{X'}\subset S'\) is generated in degree 2 and we can identify linear forms in \(R':=S'/I_{X'}\) with those in \(R_1\) vanishing on \(E\), and Gram matrices \(\calG'_k\subseteq S'_2\) with those in \(\calG_k\subseteq S_2\) of which \(E\) is contained in the kernels.

We now use the projection technique to show Theorem~\ref{thm:SurfacesOfMinimalDegree}.

\begin{proof}[Proof of Theorem~\ref{thm:SurfacesOfMinimalDegree}]
    Assume for contradiction that \(\bdl\) is a spurious local minimum in the interior. Since $\sigma_3(\bdl)$ is in the interior, the linear forms \(l_1,l_2,l_3\) do not share a common real zero on \(X\).
    By Lemma~\ref{lemma:FiniteGramNotSpuriousMinimum}, there are infinitely many Gram matrices in \(\calG_3\) associated with \(\sigma_3(\bdl)\). Since a Gram matrix associated with linear forms that do not have a common zero on $X$ is locally isolated as the Jacobian matrix has full rank, the linear series associated with these Gram matrices must have common zeros. A common zero of $\bdl$ on $X$ is necessarily a singularity of the subscheme  \(V(\sigma_3(\bdl))\) on \(X\), which only has finitely many singularities by assumption.
    Thus there exists a complex pair of points \(p,\bar{p}\in V(\sigma_3(\bdl))\) such that there are infinitely many Gram matrices in \(\calG_3\) associated with linear series vanishing at \(p\) and \(\bar{p}\).
    In particular, \(p\) and \(\bar{p}\) are smooth points on \(X\) because \(X\) can have at most one singular point (which is the cone over a rational normal curve), by the classification of varieties of minimal degree~\cite{eisenbud1987varieties}.
    Now we project away from the real subspace \(E\) spanned by \(p\) and \(\bar{p}\), and denote \(X':=\pi_E(X)\) with \(R':=\bbR[X']\).
    Here, \(X'\subseteq\bbP^{n-2}\) is again a real surface of minimal degree.
    The preimage of \(\sigma_3(\bdl)\) under the pullback map \(\pi_E^\sharp\) still lies in the interior of \(\Sigma_{X'}=\Sigma_X\cap R'_2\subset R'_2\) because \(R'_2\) passes through the interior of \(\Sigma_X\), and defines a reduced curve and has infinitely many Gram matrices in \(\calG'_3\).
    Thus by repeating the projection if needed, we may assume that either of the following situations happens:
    \begin{enumerate}
        \item \(X'=\bbP^2\), which is a contradiction because any quadratic form on \(X'\) must have a unique rank-\(3\) Gram representation;
        \item \(X'\subset\bbP^3\) is a quadric surface, which we assume is defined by a quadratic form \(Q\in S_2\).
        By Lemma~\ref{lemma:FiniteGramNotSpuriousMinimum}, if \(\bdl\) is a spurious local minimum, then \(\sigma_k(\bdl)\) has infinitely many Gram matrices. 
        As the preimage \(\sigma_k(\bdl)+\bbR\cdot Q\subset S_2\) under the quotient map \(\calG_{n+1}\to S_2\) is a line, the only possibility is that it intersects with a face of the cone of positive semidefinite quadratic forms in \(S_2\).
        Thus \(\sigma_3(\bdl)\) is also contained in a face of \(\Sigma_X\) and must lie on the boundary, which gives a contradiction.
        \qedhere
    \end{enumerate}
\end{proof}

In Example~\ref{ex:TernaryQuarticsBoundaryPoint}, we have seen that the spurious first-order stationary point on the Veronese surface is associated with infinitely Gram matrices in the obvious way: there are infinitely many representations of a binary quartic forms as a sum of three squares.
From Theorem~\ref{thm:SurfacesOfMinimalDegree}, we next show that this is exactly the reason for any boundary point to be associated with infinitely many Gram matrices. 

\begin{proposition}\label{prop:VeroneseSurfaceMinimaBinary}
    Consider the Veronese surface $X=\nu_2(\bbP^2)\subset\bbP^5$ and $k=3$.
    If $\bdl=(l_1,l_2,l_3)\in R_1^3$ lies on the boundary and has infinitely many Gram matrices associated with \(\sigma_3(\bdl)\), then $\sigma_3(\bdl)$ corresponds to a binary quartic form on $\bbP^2$ (up to a linear change of coordinates).
\end{proposition}
\begin{proof}
    As $\bdl$ is on the boundary, $l_1$, $l_2$, and $l_3$ must have a common real zero on \(X\)~\cite[Theorem 1.1]{blekherman2016sums}.
    Suppose that there exist two real points \(p,q\) in the zero locus of \(l_1,l_2,l_3\) on \(X\). 
    By changing the coordinates in \(\bbP^2\), we may assume that \(p=\nu_2([0\mathcolon0\mathcolon1])\) and \(q=\nu_2([0\mathcolon1\mathcolon0])\).
    Thus the projection \(X':=\pi_{p,q}(X)\subset\bbP^3\) is a smooth quadric surface (as shown in Figure~\ref{fig:PolytopeVeroneseProjection2RealPoints}), which implies that preimage \((\kappa')^{-1}(g)\subset S'_2\) is a line, where \(\kappa':S'_2\to R'_2\) is the canonical quotient map and \(g:=(\pi_{p,q}^\sharp)^{-1}(\sigma_3(\bdl))\in R'_2\) is the preimage of \(\sigma_3(\bdl)\) under the pullback map \(\pi_{p,q}^\sharp\).
    However, as the kernel of any Gram matrix in \(\kappa^{-1}(\sigma_3(\bdl))\) contains \(p\) and \(q\), there are infinitely many Gram matrices associated with \(g\).
    This is only possible when \((\kappa')^{-1}(g)\) intersects with the boundary of the positive semidefinite cone in \(S'_2\), in which case \(X'\) is singular.
    The contradiction with \(X'\) being smooth shows that \(l_1\), \(l_2\), and \(l_3\) can share at most one common real zero on \(X\).

    Now we assume \(p=\nu_2([0\mathcolon0\mathcolon1])\) is the unique real zero of \(l_1\), \(l_2\), and \(l_3\) on \(X\), and let \(X':=\pi_p(X)\subset\bbP^4\) (Figure~\ref{fig:PolytopeVeroneseProjection1RealPoint}).
    Since \(p\) is a real zero of \(l_1,l_2,l_3\), there exists \(l'_i\in R'_1\) corresponding to \(l_i\) for \(i=1,2,3\).
    Now we further assume that \(\sum_{i=1}^{3}(l'_i)^2\) does not have a real zero on \(X'\).
    Since it is associated with infinitely many Gram matrices, by the proof of Theorem~\ref{thm:SurfacesOfMinimalDegree} we see that \(V(\sum_{i=1}^{3}(l'_i)^2)\subset X'\) must be nonreduced.
    Consequently, the quartic form associated with it must be divisible by a square.
    By assumption, it has a unique real zero in \(\bbP^2\) so \(\nu_2^\sharp(\sigma_3(\bdl))=q^2\) for some \(q\in\bbR[x_0,x_1,x_2]_2\) that also has a unique real zero.
    Thus \(q\) has rank 2 as a quadratic form in \(x_0\), \(x_1\), and \(x_2\), which means that there exists a linear change of coordinates \(\alpha:\bbP^2\to\bbP^2\) such that \(q\circ\alpha\) is a binary quadratic form.
    Consequently, \(\sigma_3(\bdl)\) corresponds to a binary quartic form, as desired.

    It remains to discuss the case where \(\sum_{i=1}^{3}(l'_i)^2\) has a real zero \(p'\) on \(X'\), which must be a common zero of \(l'_1,l'_2,l'_3\).
    By assumption, \(p\) is the unique real zero on \(X\), so \(p'\) is in the image of the exceptional divisor of the blow-up of \(X\) at \(p\).
    In other words, \(p'\) corresponds to a real tangent direction of \(X\) at \(p\), so \(\nu_2^\sharp\circ\pi_p^\sharp(l'_i)(x_0,x_1,1)\) is an inhomogeneous quadratic polynomial in the variables \(x_0\) and \(x_1\) whose gradient at the point \((0,0)\) vanish at a common direction for each \(i=1,2,3\). 
    By changing coordinates in the variables \(x_0\) and \(x_1\), we may assume that this direction is the \(x_1\)-direction, which means the quadratic form \(\nu_2^\sharp(l_i)=\nu_2^\sharp\circ\pi_p^\sharp(l'_i)\) does not contain the monomial \(x_1x_2\) for each \(i=0,1,2\).
    Thus if we project \(X'\) away from \(p'\) through map \(\pi_{p'}\), we get a singular quadric surface \(X''\subset\bbP^3\) which is a cone over the rational normal curve (as illustrated by Figure~\ref{fig:PolytopeVeroneseProjection1RealTangent}).
    The image of \(\sum_{i=1}^{3}(l'_i)^2\) under the pullback map \(\pi_{p'}^\sharp\) is still associated with infinitely many Gram matrices, one of which we denote as \(Q\) on \(\bbP^3\).
    This can only happen as the unique singular point of \(X''\) lies in the kernel of \(Q\) since \(\codim(X'')=1\), which
implies that the quadratic forms \(\nu_2^\sharp(l_i)\) do not involve the monomial \(x_0x_2\) either, and thus \(\nu_2^\sharp(\sum_{i=1}^{3}l_i^2)\) is a binary form under change of coordinates in \(\bbP^2\).
    This completes the proof.
\end{proof}

\begin{figure}[htbp]
    \centering
    \begin{subfigure}{0.2\textwidth}
        \centering
        \includegraphics[width=3cm]{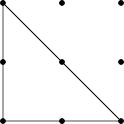}
        \caption{\(X\subset\bbP^5\)}
    \end{subfigure}
    \hspace{2mm}
    \begin{subfigure}{0.2\textwidth}
        \centering
        \includegraphics[width=3cm]{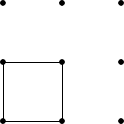}
        \caption{\(\pi_{p,q}(X)\subset\bbP^3\)}
        \label{fig:PolytopeVeroneseProjection2RealPoints}
    \end{subfigure}
    \hspace{2mm}
    \begin{subfigure}{0.2\textwidth}
        \centering
        \includegraphics[width=3cm]{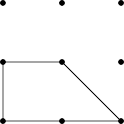}
        \caption{\(\pi_p(X)\subset\bbP^4\)}
        \label{fig:PolytopeVeroneseProjection1RealPoint}
    \end{subfigure}
    \hspace{2mm}
    \begin{subfigure}{0.2\textwidth}
        \centering
        \includegraphics[width=3cm]{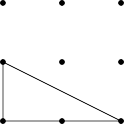}
        \caption{\(\pi_{p'}(X')\subset\bbP^3\)}
        \label{fig:PolytopeVeroneseProjection1RealTangent}
    \end{subfigure}
    \caption{Polygons corresponding to the Veronese surface and its projections}
\end{figure}

\subsection{Proof of Theorem~\ref{thm:VeroneseSurfaceMinima}}

Before proving Theorem~\ref{thm:VeroneseSurfaceMinima}, we first establish a lemma that partially extends Corollary~\ref{cor:NoSpuriousMinimaWithMoreSquares} on varieties of minimal degree.

\begin{lemma}\label{lemma:InteriorSpuriousLocalMinimaMoreSquares}
    Suppose that \(X\) is a variety of minimal degree, and any spurious second-order stationary point in \(R_1^k\) is on the boundary for \(k=\dim{X}+1\).
    Then for any \(m\ge k\), if \(\bdl\in R_1^{m}\) is a spurious second-order stationary point, then \(\sigma_m(\bdl)\) also lies on the boundary of \(\Sigma_X\).
\end{lemma}
\begin{proof}
    We prove the assertion by induction on \(m=k,k+1\dots,n+1\), where it is trivially true for the base case \(m=k\).
    For \(m+1\) squares, take any \(\bdl=(l_1,\dots,l_{m+1})\in R_1^{m+1}\) that is a spurious second-order stationary point and \(\sigma_{X,m+1}(\bdl)\notin\partial{\Sigma_X}\).
    Since \(X\) is a variety of minimal degree, this implies that \(l_1,\dots,l_{m+1}\) do not share a common real zero on \(X\)~\cite[Theorem 1.1]{blekherman2016sums}.
    Let \(L:=\sAngle{\bdl}_1=\linspan_\bbR\{l_1,\dots,l_{m+1}\}\) denote all linear forms generated by \(l_1,\dots,l_{m+1}\), and \(H_p:=\{l\in L:l(p)=0\}\) denote linear forms vanishing at a real point \(p\in X(\bbR)\). 
    We note that \(H_p\) is a hyperplane in \(L\) for any \(p\in X(\bbR)\) because \(p\) is not a common zero of \(l_1,\dots,l_{m+1}\).
    Let \(L^\vee:=\operatorname{Hom}_\bbR(L,\bbR)\) denote all linear functionals on \(L\), and the set of all hyperplanes in \(L\) is \(\bbP(L^\vee)\), which as a dimension of \(\dim_\bbR{L}-1=m\), as \(l_1,\dots,l_{m+1}\) must be linearly independent by Corollary~\ref{cor:LinearDependenceNotLocalMinimum}.
    Thus the union of all such hyperplanes \(\coprod_{p\in X(\bbR)}H_p\subset X\times\bbP(L^\vee)\) has a dimension of at most \(\dim(X)\), so the dimension of its image under \(X\times\bbP(L^\vee)\to\bbP(L^\vee)\) also does not exceed \(\dim(X)<m\).
    Consequently, there exists a hyperplane \(H\in\bbP(L^\vee)\) such that \(H\neq H_p\) for any \(p\in X(\bbR)\).

    Take an orthonormal basis of \(H\), i.e., \(h_1,\dots,h_m\) spanning \(H\) such that \(h_i=\sum_{j=1}^{m+1}c_{i,j}l_{j}\), for some orthonormal vectors \(c_i=(c_{i,1},\dots,c_{i,m+1})\in\bbR^{m+1}\), \(i=1,\dots,m\).
    Note that \(\sum_{i=1}^m h_i^2\notin\partial{\Sigma_X}\).
    By the induction hypothesis, \((h_1,\dots,h_m)\) is not a spurious second-order stationary point.
    By completing \(c_1,\dots,c_m\) to an orthonormal basis \(c_1,\dots,c_{m+1}\) of \(\bbR^{m+1}\), let \(h_{m+1}:=\sum_{j=1}^{m+1}c_{m+1,j}l_{j}\in R_1\).
    Corollary~\ref{cor:NoSpuriousMinimaWithMoreSquares} implies that \((h_1,\dots,h_{m+1})\) is not a spurious second-order stationary point, and thus neither is \((l_1,\dots,l_{m+1})\) by Corollary~\ref{cor:LocalMinimaOrthogonalTransform}.
    This completes the induction step.
\end{proof}

The following lemma contains some key technical claims for Theorem~\ref{thm:VeroneseSurfaceMinima}.
\begin{lemma}\label{lemma:VeroneseZerosStructure}
Let $X=\nu_2(\bbP^2)\subset\bbP^5$ be the Veronese surface and suppose $\bdl:=(l_1,l_2,l_3)$ is a tuple of three linearly independent linear forms defining an ideal $I:=\sAngle{\bdl}$ in the homogeneous coordinate ring $R$ of $X$. If $l_1,l_2,l_3$ have no common real zeroes on $X$ then the following statements hold:
\begin{enumerate}
\item The zero set \(V(l_1,l_2,l_3)\subset X\) is finite.
\label{lemma:Veronese:claim1}
\item The saturation $I_{\operatorname{sat}}$ of the ideal $I$ is radical and defines either the empty set or a conjugate pair $\{p,\bar{p}\}$ of complex  points on $X$.
\label{lemma:Veronese:claim2}
\item The degree-$2$ part of the ideal $I$ coincides with its saturation.
\label{lemma:Veronese:claim3}
\item If the forms $l_1,l_2,l_3$ have as common zeroes the conjugate pair $\{p,\bar{p}\}$ then the image of the evaluation map $\operatorname{ev}_p$ is a one-dimensional vector space, where
\[
    \begin{aligned}
    \operatorname{ev}_p:&&\syzygy_1(\bdl)&\rightarrow \bbC^3\\
    &&(h_1,h_2,h_3)&\mapsto(h_1(p),h_2(p),h_3(p)).
    \end{aligned}
\]
\label{lemma:Veronese:claim4}
\end{enumerate}
\end{lemma}

\begin{proof}
\ref{lemma:Veronese:claim1} Via the pullback map \(\nu_2^\sharp\), the linear forms \(l_1,l_2,l_3\) correspond to linearly independent quadratic forms $a_1,a_2,a_3\in\bbR[x_0,x_1,x_2]_2$ having no common real zeroes on $\bbP^2$. If $a_1,a_2,a_3$ had infinitely many common zeroes (over $\bbC$), then the zeros would form a curve $C\subset\bbP^2$ and the form $b\in\bbC[x_0,x_1,x_2]$ defining $C$ would be a common factor of all the $a_1,a_2,a_3$. Since the $a_1,a_2,a_3$ are real, they can also be divided by the conjugate $\bar{b}$. If $b\neq\bar{b}$, then every of $a_1,a_2,a_3$ would be a constant multiple of the product $b\cdot\bar{b}$ contradicting the fact that the $a_i$ are linearly independent. We conclude that $b=\bar{b}$ so $C$ is a real line in $\bbP^2$, contradicting that $a_1,a_2,a_3$ have no common real zeros.

\ref{lemma:Veronese:claim2} Since the Veronese surface is a variety of minimal degree it is a \emph{small scheme} which means that for every projective subspace $\Lambda\subseteq \bbP^5$ such that $X\cap \Lambda$ is zero-dimensional we know that the length of the scheme $\Lambda\cap X$ is at most ${\rm dim}(\Lambda)+1$~\cite{eisenbud2006small}. In our case the projective subspace $\Lambda$ defined by the linear forms $l_1,l_2,l_3$ is isomorphic to $\bbP^{5-3}=\bbP^2$ and therefore the degree of the scheme $X\cap \Lambda$ is at most $3$. We claim that the degree of $X\cap \Lambda$ is either $0$ or $2$ and the scheme $X\cap \Lambda$ is reduced.
It is immediate that it cannot have degree one since otherwise it would be a real point. Next we show that the degree cannot be $3$. If $X\cap \Lambda$ had degree three then we will show that it must contain at least one real point which is impossible since the forms have no common real zero. More precisely, we can have the following cases.
\begin{itemize}
\item $X\cap\Lambda$ is supported at three distinct points, forcing it to be reduced. The conjugation action implies at least one of these points is real. Or
\item $X\cap \Lambda$ is supported at exactly two distinct points of multiplicities $1$ and $2$. The points cannot be exchanged by conjugation since they have different multiplicities so one of them is real. Or
\item $X\cap \Lambda$ is supported at exactly one point which must therefore be real.
\end{itemize}
We thus conclude that $X\cap \Lambda$ either has degree zero, or it is supported at a conjugate pair of points $\{p,\bar{p}\}$ and it has degree $2$, which implies that it is reduced.
This is equivalent to saying that the unique saturated ideal $I_{\operatorname{sat}}$ defining the points is radical.

\ref{lemma:Veronese:claim3} The claim is equivalent to showing that the dimension of $(R/I)_2$ is two. 
By part~\ref{lemma:Veronese:claim1} we know that the ideal $I$ has height two in $R$ and thus by prime avoidance contains a regular sequence of linear forms $(g_1,g_2)$. 
By Bertini's theorem~\cite{jouanolou1983theoremes}, $\sAngle{g_1,g_2}$ defines a reduced subscheme of four points of $X$ denoted as $Y:=\{r_1,r_2,p,\bar{p}\}$. By smallness of $X$ these points are furthermore projectively independent. 

Take any $g_3\in R_1$ so that $(g_1,g_2,g_3)=I$ and consider the graded exact sequence
\[0\rightarrow K\rightarrow R/(g_1,g_2)[-1]\xrightarrow{\cdot g_3} R/(g_1,g_2)\rightarrow R/I\rightarrow0\]
where the middle map is multiplication by $g_3$. By construction the degree $d=2$ part of $K$ consists of the functions $f$ on $Y$ of degree $d-1$ which vanish when multiplied by $g_3$. Since $g_3$ does not vanish at either $r_1$ or $r_2$ and vanishes at $p,\bar{p}$, this is equivalent to $f$ satisfying $f(r_1)=f(r_2)=0$. Since $Y$ consists of linearly independent points the space of such functions is $2$-dimensional, having codimension $2$ in $R/(g_1,g_2)_1$. We conclude that the image of the multiplication by $g_3$ is two-dimensional so $\dim_\bbR{(R/I)_2}=\dim_\bbR{(R/(g_1,g_2))_2}-2=4-2=2$ as claimed.

To prove~\ref{lemma:Veronese:claim4}, note that the statement is about vector space dimensions and that it does not depend on the chosen generating set for the ideal $I$. We can therefore assume that $l_1=g_1$ and $l_2=g_2$ form a regular sequence. We let $R':=R/(l_1,l_2)$ and claim that the  following statements hold:
\begin{enumerate}[label=(\Alph*)]
\item If $\pi:R_1^3\rightarrow R_1$ denotes the projection onto the last component and $q:R_1\rightarrow R'_1$ is the quotient map then the composition $q\circ \pi$ defines an injective map from $\syzygy_1(\bdl)/T$ to $R'_1$, where $T\subset\syzygy_1(\bdl)$ is the subspace spanned by the trivial Koszul syzygies \((-l_3,0,l_1)\), \((0,-l_3,l_2)\), and \((l_2,-l_1,0)\).
\label{lemma:Veronese:claim4A}
\item The image of $q\circ \pi$ is the set of linear forms vanishing at the points $r_1,r_2$ so it has codimension two in $R'_1$.
\label{lemma:Veronese:claim4B}
\item Let $W\subset\syzygy_1(\bdl)$ denote the subspace of syzygies whose components all vanish at $p$. The image of $W$ under $q\circ \pi$  are precisely the functions in $R'_1$ vanishing at $r_1,r_2,p$ so $W/T$ has codimension three in $R'_1$.
\label{lemma:Veronese:claim4C}
\end{enumerate}
Verifying these claims completes the proof since $\operatorname{ev}_p$ annihilates $T$ and thus descends to a map $\operatorname{ev}_p:\syzygy_1(\bdl)/T\rightarrow \bbC^3$, which by part~\ref{lemma:Veronese:claim4B} has an image that is at most two dimensional. Moreover $\operatorname{ev}_p$ annihilates the subspace $W/T$ which by part~\ref{lemma:Veronese:claim4C} has codimension three. We conclude that the image of the evaluation map is one-dimensional proving~\ref{lemma:Veronese:claim4}.
All that remains is to verify the claims~\ref{lemma:Veronese:claim4A}-\ref{lemma:Veronese:claim4C}. 

\ref{lemma:Veronese:claim4A} The composition $q\circ \pi$ defines a linear map from $\syzygy_1(\bdl)$ to $R'_1$ which maps every trivial syzygy to zero. So it descends to a linear map from $\syzygy_1(\bdl)/T$ to $R'_1$. If a tuple $(h_1,h_2, h_3)$ maps to zero in $R'_1$, then $h_3=\sum_{i=1}^{2}c_i\bdl_i$ for some $c_1,c_2\in\bbR$ so we can make the last component of $(h_1,h_2,h_3)$ zero by subtracting trivial syzygies. If the last component of a syzygy is zero then the first $2$ components are a syzygy for $(l_1,l_2)$ so these components are in the span of the Koszul syzygies on $l_1,l_2$ by our assumption that they form a regular sequence. This proves that the composition is injective as claimed.
\ref{lemma:Veronese:claim4B} Take any syzygy $(h_1,h_2,h_3)$ and we know that $h_3l_3=0$ in $R'_1$. Since $R'$ is the coordinate ring of a reduced set of four independent points, vanishing at the points $r_1,r_2$ imposes independent conditions on the linear form $h_3$. We conclude that the image of the space $\syzygy_1(\bdl)/T$ is equal to the space of such forms and thus has codimension two in $R'_1$.
\ref{lemma:Veronese:claim4C} Assume $h_3(p)=0$ and $(h_1,h_2,h_3)$ is a linear syzygy so the equation $h_1l_1+h_2l_2+h_3l_3=0$ holds identically in $R$. Pulling back to $\bbP^2$ via the Veronese we can assume that this is an equality of quadratic forms in $\bbP^2$. Since $l_1,l_2$ define the point locally around $p$, their gradients at $p$ are linearly independent. This implies that both $h_1$ and $h_2$ must vanish at $p$. This observation implies that $W$ is mapped via $q\circ \pi$ onto the functions in $R'_1$ vanishing at $r_1,r_2$ and $p$, proving the final claim because the points of $Y$ are linearly independent. 
\end{proof}

We are now ready to prove Theorem~\ref{thm:VeroneseSurfaceMinima}.

\begin{proof}[Proof of Theorem~\ref{thm:VeroneseSurfaceMinima}]
    For the first assertion, by Lemma~\ref{lemma:InteriorSpuriousLocalMinimaMoreSquares}, we only need to consider \(k=3\).
    Assume for contradiction that \(\bdl=(l_1,l_2,l_3)\in R_1^k\) is a spurious second-order stationary point and \(l_1,l_2,l_3\) do not share any common real zero on \(X\).
    Then \(l_1,l_2,l_3\) must have a common zero on \(X\), or otherwise they would generate \(R_2\) and fail the first-order optimality condition.
    Part~\ref{lemma:Veronese:claim2} of Lemma~\ref{lemma:VeroneseZerosStructure} tells us that the common zero consists of a conjugate pair \(\{p,\bar{p}\}\) of complex points on \(X\), and
    part~\ref{lemma:Veronese:claim3} says that any quadratic form on \(X\) that vanish on the common zeros \(V(l_1,l_2,l_3)\) must be generated by the linear forms \(l_1,l_2,l_3\).
    This means that any \(g\perp\sAngle{\bdl}_2\), the linear functional \(\pAngle{g}{\cdot}\) can be identified with the evaluation at (some representation of) \(p\) or \(\bar{p}\), i.e., 
    \[
        \pAngle{g}{f}=b_1\operatorname{re}(f(p))+b_2\operatorname{im}(f(p)),\quad f\in R_2,
    \]
    for some \(b_1,b_2\in\bbR\) that are not both 0.

    Now use the evaluation map \(\operatorname{ev}_p\) in part~\ref{lemma:Veronese:claim4} of Lemma~\ref{lemma:VeroneseZerosStructure}, which has a one-dimensional image in \(\bbC^3\).
    This means that we can find \(c_1,c_2\in\bbC\) such that \(h_i(p)=c_ih_3(p)\) for \(i=1,2\) for all syzygies \(\bdh\in\syzygy_1(\bdl)\).
    Thus for any \(g\perp\sAngle{\bdl}_2\),
    \[
        \pAngle{g}{\sigma_3(\bdh)}=b_1\cdot\mathrm{re}(C\cdot h_3^2(p))+b_2\cdot\mathrm{im}(C\cdot h_3^2(p)),\quad C=1+c_1^2+c_2^2\in\bbC.
    \]
    If \(C=0\), then \(\pAngle{g}{\sigma_3(\bdh)}=0\) but \(h_3(p)\neq0\) so \(g\not\perp\sAngle{\bdh}_2\), which implies that \(\bdl\) is not a spurious second-order stationary point by Theorem~\ref{thm:CharacterizationSpuriousMinima}.
    Otherwise we may assume \(C\neq0\), in which case we may rewrite
    \[
        \pAngle{g}{\sigma_3(\bdh)}=b'_1\operatorname{re}(h_3^2(p))+b'_2\operatorname{im}(h_3^2(p))
    \]
    for some real numbers 
    \[
        \begin{bmatrix}b'_1 \\ b'_2 \end{bmatrix}=
        \begin{bmatrix}\operatorname{re}(C) & -\operatorname{im}(C) \\
            \operatorname{im}(C) & \operatorname{re}(C)
        \end{bmatrix}\cdot\begin{bmatrix} b_1 \\ b_2 \end{bmatrix}\in\bbR^2\setminus\{(0,0)\}.
    \]
    This implies that \(\bdh\mapsto \pAngle{g}{\sum_{i=1}^{3}h_i^2}\) is indefinite on \(\syzygy_1(\bdl)\), so that \(\bdl\) is not a spurious second-order stationary point again by Theorem~\ref{thm:CharacterizationSpuriousMinima}.

    Now for the second assertion, we only need to consider boundary points \(\bdl\) by the first assertion.
    By Lemma~\ref{lemma:FiniteGramNotSpuriousMinimum} it is associated with infinitely many Gram matrices.
    In this case, Proposition~\ref{prop:VeroneseSurfaceMinimaBinary} shows that \(\sigma_3(\bdl)\) can be identified with a binary quartic form on \(\bbP^2\) under the pullback map \(\nu_2^\sharp\). 
    This implies that each component \(l_1\), \(l_2\), and \(l_3\) can also be identified with quadratic forms on \(\bbP^2\) in the same two variables, say \(x_0\) and \(x_1\).
    Corollary~\ref{cor:LinearDependenceNotLocalMinimum} ensures that \(l_1,l_2\), and \(l_3\) are linearly independent, which implies that their span consists of those corresponding to quadratic forms in \(x_0\) and \(x_1\) on \(\bbP^2\), and the syzygies \(\syzygy_1(\bdl)\) do not involve monomials divisible by \(x_2^2\).
    Thus the argument in Example~\ref{ex:TernaryQuarticsBoundaryPoint} shows that the only possible candidate for \(\bar{f}\) to make \(\bdl\) a spurious second-order stationary point is \(\bar{f}=\sigma_3(\bdl)-\epsilon(\nu_2^\sharp)^{-1}(x_2^4)\in R_2\) for some \(\epsilon>0\), which does not make \(\bdl\) a spurious local minimum.
\end{proof}

\section{Spurious local minima on general varieties}

In this section, we expand our horizon to varieties of higher degree.
We first show the existence of spurious local minima in the interior on \(m\)-dimensional Veronese varieties \(\nu_2(\bbP^m)\).
Then we show that all such spurious local minima are uncommon by bounding its dimension, and propose a conceptual algorithmic framework to avoid spurious local minima in the interior.

\subsection{Spurious local minima in the interior}

We begin our discussion on varieties of higher degree with the following example.
\begin{example}\label{ex:QuarticsSpuriousMinima}
    Let \(X=\nu_2(\bbP^m)\) where points in \(\bbP^m\) are denoted by \([x_0\mathcolon\cdots\mathcolon x_m]\). 
    Pick \(\bdl=(l_1,\dots,l_k)\in R_1^k\) such that \(\nu_2^\sharp(l_1)=x_0^2+x_1^2\), \(\nu_2^\sharp(l_2)=x_2^2\), \(\nu_2^\sharp(l_3)=x_2x_3\), \(\dots\), \(\nu_2^\sharp(l_k)=x_m^2\) where \(k=1+\binom{m}{2}\).
    In plain words, \(l_1,\dots,l_k\) correspond to the quadratic form \(x_0^2+x_1^2\) and all quadratic monomials in the variables \(x_2,\dots,x_m\).
    Note that their only common zero is the complex pair \(p:=(1,\pm\sqrt{-1},0,\dots,0)\in\bbP^m\), so the sum of squares \(\sum_{i=1}^{k}l_i^2\) is indeed positive on \(X(\bbR)\).
    When \(m\ge10\), it is easy to check that \(\binom{k}{2}\ge\dim_\bbR(R_2)=\binom{m+4}{4}\) so \(k\ge\Pythagoras(X)\) by the first bound in~\cite{blekherman2021sums}. 
    We claim that 
    \begin{enumerate}
        \item \(\sum_{i=1}^{k}l_i^2\in\interior{\Sigma_X}\), and 
        \item there exists \(\bar{f}\in R_2\) such that \(\bdl\) is a spurious local minimum.
    \end{enumerate}
    We first show the assertion (ii).
    For any syzygy \(\bdh=(h_1,\dots,h_k)\in\syzygy_1(\bdl)\), we will show \(h_i(p)=0\) for \(i=1,\dots,k\) as the point evaluation at (any representation of) \(p\) defines a desired \(g\in R_2\) because \(\sAngle{\bdl}_2\cap\Sigma_X^*=0\) in this case.
    By some slight abuse of notation, we refer to \(h_1,\dots,h_k\) as quadratic forms on \(\bbP^m\) with variables \(x_0,x_1,\dots,x_m\). 
    It is easy to see that \(h_1\) does not involve \(x_0\) or \(x_1\), as other polynomials \(l_ih_i\) have degrees in \(x_0\) or \(x_1\) at most 2, \(i=2,\dots,k\).
    For the rest, if \(h_i\) involves \(x_0\) or \(x_1\), then the product \(l_ih_i\) can only be cancelled by multiplies of \(x_0^2+x_1^2\), which means that \(h_i\) is divisible by \(x_0^2+x_1^2\).
    Therefore, \(\syzygy_1(\bdl)\subseteq\sAngle{\bdl}_1^k\), which implies \(h_i\) vanishes at \(p\) for all \(i=1,\dots,k\).
    By Theorem~\ref{thm:ExistenceSpuriousMinima}, we know that \(\bdl\) is a spurious second-order stationary point.
    Moreover, \(\bdh\in\syzygy_1(\bdl)\) implies that \(h_i\in\sAngle{\bdl}_1\) for each \(i=1,\dots,k\), so by Lemma~\ref{lemma:LocalMinimumSufficientCondition}, we know that \(\bdl\) is in fact a spurious local minimum.

    It remains to show the assertion (i), which can be done by a standard perturbation argument as follows.
    We claim that for any single square of a monic quadratic monomial \(q\) in \(x_0,x_1,\dots,x_n\) and \(\epsilon\le1\), \(\sigma_k(\bdl)-\epsilon\cdot q^2\in\Sigma_X\).
    \begin{itemize}
        \item If \(q\) only involves \(x_2,\dots,x_n\), then it is obvious as \(q\) is one of \(l_2,\dots,l_k\).
        \item If \(q\) only involves \(x_0\) and \(x_1\), then any \(\epsilon\le1\) also works because \(l_1^2-\epsilon q^2=(x_0^2+x_1^2)^2-q^2=x_0^4+2x_0^2x_1^2+x_1^4-q^2\) is again a sum of squares for each possibility \(q=x_0^2\), \(q=x_1^2\), and \(q=x_0x_1\).
        \item If \(q=x_ix_j\) where \(i=0,1\) and \(j=2,\dots,k\), note that \((x_ix_j)^2=(x_i^2)(x_j^2)\), so a Gram matrix of \(\sigma_k(\bdl)-\epsilon\cdot q^2\) can be written as (for example \(i=0\) and \(j=k\))
        \begin{equation*}
            \begin{bmatrix}
                1 & 0 & 0 &  & \cdots &  & 0 & 0  & \cdots & -\epsilon/2\\
                0 & 2 & 0 &  & \cdots &  & 0 & 0  & \cdots  & 0\\
                0 & 0 & 1 &  & \cdots &  & 0 & 0  & \cdots  & 0\\
                &   &   & 0 &   &   &  &  &  &  \\
                \vdots & \vdots & \vdots &   & \ddots &  & \vdots & \vdots & & \vdots \\
                &   &   &   &   & 0 &  &  &  &  \\
                0 & 0 & 0 &  & \cdots &  & 1 & 0 & \cdots & 0\\
                0 & 0 & 0 &  & \cdots &  & 0 & 1 & \cdots & 0\\
                \vdots & \vdots & \vdots &  & \cdots &  & \vdots & \vdots & \ddots & 0\\
                -\epsilon/2 & 0 & 0 &  & \cdots &  & 0 & 0 & \cdots & 1
            \end{bmatrix},
        \end{equation*}
        where the first three columns/rows correspond to monomials \(x_0^2\), \(x_0x_1\), and \(x_1^2\), and the bottom right block corresponds to quadratic monomials in \(x_2,\dots,x_n\).
        This Gram matrix is positive semidefinite for \(\epsilon\le2\) because of the diagonal dominance.
        Therefore, \(\sigma_k(\bdl)-\epsilon\cdot q^2\) is still a sum of squares.
    \end{itemize}
    Now notice that for any monic monomial \(ab\), \(a,b\in\bbR[x_0,\dots,x_n]_2\), 
    \[
        \sigma_k(\bdl)\pm \frac{2}{3}ab=\frac{1}{3}\left((\sigma_k(\bdl)+(a\pm b)^2)+(\sigma_k(\bdl)-a^2)+(\sigma_k(\bdl)-b^2)\right)
    \]
    is again a sum of squares.
    Then by the convexity of \(\Sigma_X\) and that monic monomials span \(\bbR[x_0,\dots,x_n]_4\), we conclude that \(\sigma_k(\bdl)\in\interior{\Sigma_X}\).
\end{example}

Nevertheless, the locus of spurious local minima in the interior is small as quantified in Theorem~\ref{thm:SmallLocusSpuriousMinima}.
To be more precise, given a smooth, totally real variety \(X\subseteq\bbP^n\), let \(\Delta\subseteq R_2\) denote the Zariski closure of all quadratic forms that are singular at some point of \(X\).
It is known that \(\codim(\Delta\cap\interior{\Sigma_X})\ge2\) because it is contained in the singular locus of the discriminant of the second Veronese re-embedding of \(X\)~\cite[Theorem 2.2]{blekherman2021sums}.

\begin{proof}[Proof of Theorem~\ref{thm:SmallLocusSpuriousMinima}]
    Let \(C\subset R_1^k\) denote the set of all linear forms having a common complex zero on \(X\).
    As \(k\ge r(X)\), \(\bdl\in R_1^k\) is a spurious local minimum only if \(\bdl\in C\).
    Note that \(\sigma_k(\bdl)\in\Delta\cap\interior{\Sigma_X}\) if \(\bdl\in C\).
    Thus the codimension of the Zariski closure of \(\sigma_k(C)\) is at least \(\codim(\Delta\cap\interior{\Sigma_X})\ge2\).
\end{proof}

Arguments based on low dimensionality of spurious local minima have appeared in many smoothed analyses of Burer-Monteiro methods~\cite{boumal2016nonconvex,pumir2018smoothed,boumal2020deterministic,cifuentes2021burer,cifuentes2022polynomial}.
In our case, $X$ is typically not a complete intersection, which fails the regularity assumption in these works.

\subsection{A restricted path algorithm}

Given the existence of spurious local minima in the interior, as shown in Example~\ref{ex:QuarticsSpuriousMinima}, one may want to avoid them generically using the fact that \(\codim{(\Delta\cap\interior{\Sigma_X})}\ge2\) where \(\Delta\) is the Zariski closure of all quadratic forms that are singular at some smooth point of \(X\). 
This can be achieved if we algorithmically restrict our search path in \(R_2\) to be close to the line connecting our starting point and the target.
In fact, the algorithm can be stated for the more general formulation~\eqref{eq:SOSRelaxationProblem}, which specializes to~\eqref{eq:SOSMinimizationProblem} by
picking any tuple of linear forms \(\bdl^0\) and setting \(f:=\sigma_k(\bdl^0)\), \(g:=f-\bar{f}\), \(\ubar{v}=0\), and \(\bar{v}=1\).
In this case, whenever we solve Problem~\eqref{eq:SOSRelaxationProblem} with \(v^*=1\), the corresponding solution \(\bdl^*\in R_1^k\) is a certificate for \(\bar{f}\) being a sum of \(k\) squares.

The benefit of the alternative formulation~\eqref{eq:SOSRelaxationProblem} is that it naturally leads to a one-dimensional search algorithm in \(v\).
With some initial value \(\ubar{v}\), such that we already have a sum-of-squares representation for \(f-\ubar{v}\cdot g\), we can search for the maximum \(v^*\) by iteratively increasing its value by a small step \(u>0\), i.e., \(v_j:=v_{j-1}+u\) for \(j=1,2,\dots\) with \(v_0:=\ubar{v}\), and solving~\eqref{eq:SOSMinimizationProblem} with a temporary target being \(\bar{f}=\bar{f}_j=f-v_j\cdot g\) until a positive distance \(\nVert{\sigma_k(\bdl^j)-\bar{f}_j}>0\) is returned.
We summarize this procedure in Algorithm~\ref{alg:AvoidSpuriousStationaryPoints}, where any descent method in step~\ref{alg:AvoidSpuriousStationaryPoints:Subproblem} is assumed to give a monotone decreasing objective distance in~\eqref{eq:SOSMinimizationProblem}.
The correctness of the algorithm is then checked in Proposition~\ref{prop:RestrictedPath}, where it is worth noting that the argument remains valid even if there are spurious stationary points on the boundary, since the algorithm may restrict the search to the interior of \(\Sigma_X\).

\begin{algorithm}[ht]
    \caption{Restricted Path for Problem~\eqref{eq:SOSRelaxationProblem}}
    \label{alg:AvoidSpuriousStationaryPoints}
    \begin{algorithmic}[1]
    \Require{quadratic forms \(f,g\), a step size \(u>0\), a number \(k\ge r(X)\)}
    \Require{an initial value \(v_0=\ubar{v}\in\bbR\), and linear forms \(\bdl^0\in R_1^k\) such that \(f-v_0g=\sigma_k(\bdl^0)\in\interior{\Sigma_X}\)} 
    \Ensure{a near-optimal value \(v'>v^*-u\) and linear forms \(\bdl'\) such that \(\sigma_k(\bdl')=f-v'g\in\interior{\Sigma_X}\)}
        \State{set \(j\leftarrow 0\)}
        \Repeat
        \State{store \(v':=v^j\), \(\bdl':=\bdl^j\), and update \(j\leftarrow j+1\) and \(v^j:=v^{j-1}+u\)}
        \State{solve~\eqref{eq:SOSMinimizationProblem} with \(\bar{f}=\bar{f}_j:=f-v_jg\), starting with \(\bdl^{j-1}\in R_1^k\) and using a descent method}\label{alg:AvoidSpuriousStationaryPoints:Subproblem}
        \State{collect the solution \(\bdl^j\in R_1^k\)}
        \Until{\(v'>\bar{v}-u\) or the distance \(\nVert{\sigma_k(\bdl^j)-\bar{f}_j}>0\)}
    \end{algorithmic}
\end{algorithm}

\begin{proposition}\label{prop:RestrictedPath}
    Fix \(g\in R_2\). 
    For a generic \(f\in R_2\) such that \(f-v_0g\in\interior{\Sigma_X}\)
    and any suboptimality tolerance \(\theta>0\), 
    there exists \(U>0\) such that for any step size \(u<U\), 
    Algorithm~\ref{alg:AvoidSpuriousStationaryPoints} will return a near-optimal value \(v'>v^*-\theta\) and the associated linear forms \(\bdl'\) such that \(f-v'g=\sigma_k(\bdl')\) within finitely many iterations.
\end{proposition}
\begin{proof}
    Without loss of generality, we may assume \(\nVert{g}=1\).
    Note that in the \(j\)-th iteration, the distance \(\nVert{f-v_jg-\sigma_{k}(\bdl^{j-1})}=u\nVert{g}=u\) implies that the descent algorithm will only search for quadratic forms within distance \(u\) of the line segment \(\Lambda:=\{f+vg:v_0\le v\le v^*-\theta\}\subset f+\bbR g\).
    Thus, it suffices to choose any radius \(u>0\) such that for any \(v\in[v_0,v^*-\theta]\), the ball centered at \(v\) with radius \(u\) does not intersect with (1) the locus of stationary points in the interior, or with (2) the boundary points.

    For (1), since \((f+\bbR g)\cap\clos(\Delta\cap\interior{\Sigma_X})=\varnothing\) for generic \(f\in R_2\) by Theorem~\ref{thm:SmallLocusSpuriousMinima}, we know any radius \(u\) smaller than \(U_1:=\inf\{\nVert{h-h'}:h\in\Lambda,\ h'\in\clos(\Delta\cap\interior{\Sigma_X})\}>0\) satisfies the requirement.
    For (2), the construction can be slightly more explicit: by assumption \(f-v_0g\in\interior{\Sigma_X}\), so there exists \(\epsilon>0\) such that the ball centered at \(f-v_0g\) with radius \(\epsilon\) is contained in \(\interior\Sigma_X\).
    Since \(\Sigma_X\) is convex, any ball with radius \(U_2:=\frac{\theta}{v^*-v_0}\epsilon\) centered at \(v\in[v_0,v^*-\theta]\) is also contained in \(\interior\Sigma_X\).
    Therefore, setting \(U:=\min\{U_1,U_2\}\), we see that step~\ref{alg:AvoidSpuriousStationaryPoints:Subproblem} does not encounter any spurious (first-order) stationary points for any step size \(u<U\), and the algorithm terminates within \(\lceil\frac{v^*-\theta-v_0}{u}\rceil<\infty\) iterations.
\end{proof}

We remark that the formulation~\eqref{eq:SOSRelaxationProblem} extends to other interesting applications including sum-of-squares relaxation for polynomial optimization.
For example, suppose we want to find a lower bound of an inhomogeneous polynomial function \(F\in\bbR[x_1,\dots,x_n]\) by writing \(F-v\) as a sum of squares of polynomials in \(\bbR[x_1,\dots,x_n]\) for some \(v\in\bbR\).
It is of great interest to know the largest value \(v^*\) for \(F-v^*\) to be a sum of squares, as \(v^*\) provides a potentially good approximation of the minimization value \(\min_{x\in\bbR^n}F(x)\).
If \(\deg(F)= 2d\), we may approach this task by homogenizing \(F-v\) as \(x_0^{2d}\cdot F(x_1/x_0,\dots,x_n/x_0)-vx_0^{2d}\), which can be identified with a quadratic form \(f-vg\) on the Veronese variety \(X:=\nu_d(\bbP^n)\).
Now Algorithm~\ref{alg:AvoidSpuriousStationaryPoints:Subproblem} and Proposition~\ref{prop:RestrictedPath} still apply to this problem by setting \(\bar{v}=+\infty\) in~\eqref{eq:SOSRelaxationProblem}, assuming that we have the knowledge of some \(\ubar{v}\in\bbR\) and \(\bdl^0\in R_1^k\) such that \(f-\ubar{v}g=\sigma_k(\bdl^0)\) as the starting point.

\section{Numerical Experiments}

In this section, we report numerical experiments that illustrate our observations, and those that lead to further practical applications of~\eqref{eq:SOSMinimizationProblem}.
We implement\footnote{Code Access: \url{https://github.com/shixuan-zhang/LowRankSOS.jl}} Algorithm~\ref{alg:AvoidSpuriousStationaryPoints} in \texttt{Julia v1.6} and use the package \texttt{NLopt v1.0}~\cite{johnson2007NLopt} for its limited-memory BFGS (LBFGS) algorithm~\cite{liu1989limited} to solve Step~\ref{alg:AvoidSpuriousStationaryPoints:Subproblem}.
It is worth noting that in all of our experiments, it suffices to use directly the maximum step size $u=1$ without encountering spurious stationary points.
The reported computational times are based on a 3.7 GHz CPU with 32 GB RAM.

We first run the experiments where \(X\subset\bbP^n\) is an \(m\)-dimensional rational normal scroll for some \(m\ge2\).
Let \((n_1,\dots,n_m)\) be the heights of the Lawrence prism, as defined in~\eqref{eq:LawrencePrism}, and \(n+1=m+\sum_{i=1}^{m}n_i\).
As \(X\) has minimal degree, we know that \(\Pythagoras(X)=m+1\) and consider \(k=m+1,m+2,\) and \(m+3\) for comparison.
To produce a target \(\bar{f}\), we randomly generate a tuple of linear forms \(\bdl^{\rm targ}\in R_1^{n+1}\) using the standard normal distribution, and take its sum of squares \(\bar{f}=\sigma_{n+1}(\bdl^{\rm targ})\in R_2\).
Then for each different \(k\), we independently generate another tuple \(\bdl^{\rm init}\in R_1^k\) and start the LBFGS algorithm with it.
To avoid numerical issues, normalization is taken so \(\nVert{\bar{f}}=1\) and \(\nVert{\bdl^{\rm init}}=1\) in their norms on \(R_2\) and \(R_1^k\), respectively.
We also compare our nonconvex low-rank formulation against the standard semidefinite programming (SDP) formulation.
To be precise, recall from Section~\ref{sec:SurfacesOfMinimalDegree} that \(\calG_{n+1}\subset S_2\) is the set of positive semidefinite matrices, with the projection \(\kappa:S_2\to R_2\), and consider the following minimization in terms of the Gram matrix:
\begin{equation}\label{eq:SDPMinTrace}
    \begin{aligned}
        \min_{G\in\calG_{n+1}}\quad & \operatorname{tr}(G)\\
        \mathrm{s.t.}\quad & \kappa(G)=\sigma_{n+1}(\bdl^{\rm targ})\text{ in }R_2.
    \end{aligned}
\end{equation}
Any solution \(G\) satisfying the constraint in~\eqref{eq:SDPMinTrace} is a (positive semidefinite) Gram matrix that certifies the target as a sum of squares (see e.g.,~\cite[Chapter 3.1]{blekherman2012semidefinite}), but there is generally no guarantee on its rank. 
We thus choose our objective function in~\eqref{eq:SDPMinTrace} to be the trace (or equivalently, the nuclear norm) of the Gram matrix, due to its common use for finding low-rank solutions in matrix completion problems~\cite{koltchinskii2011nuclear,davenport2016overview,chi2019nonconvex}.
This SDP problem can be solved using an interior-point (IP) method~\cite{helmberg1996interior}, which we mainly access through the \texttt{CSDP} solver package~\cite{borchers1999csdp}. 
We also use the following SDP solvers for comparison: \texttt{SCS}~\cite{odonoghue2016conic}, \texttt{Hypatia}~\cite{coey2022solving}, and \texttt{Clarabel}~\cite{clarabel2024}.
The experiment procedure is then repeated 100 times and the results are summarized in Table~\ref{tab:ScrollResults}.
\begin{table}[htbp]
    \centering
    \begin{tabular}{cccc|cccc|cc}
        \hline
        \multicolumn{4}{c|}{Scroll Data} & \multicolumn{4}{c|}{Low-rank Form.} & \multicolumn{2}{c}{SDP Form.}\\
        \hline
        Heights & $m$ & $n$ & $\;k\;$ & Succ. & Unfin. & Spur. & Time (s) & \makecell{Min.\ Rank/\\Solver} & \makecell{Min.\ Time (s)/\\Solver} \\
        \hline
        \multirow{3}*{(5,10)} & \multirow{3}*{2} & \multirow{3}*{16} 
                              & 3 & 100 & 0 & 0 & 0.024 &
        \multirow{3}*{\makecell{4/\\\texttt{Hypatia,CSDP}}} & \multirow{3}*{\makecell{0.045/\\\texttt{SCS}}}\\
                              & & & 4 & 100 & 0 & 0 & 0.024 \\
                              & & & 5 & 100 & 0 & 0 & 0.025 \\
        \hline
        \multirow{3}*{(10,15)} & \multirow{3}*{2} & \multirow{3}*{26} 
                               & 3 & 100 & 0 & 0 & 0.069 &
        \multirow{3}*{\makecell{5/\\\texttt{Hypatia,CSDP}}} & \multirow{3}*{\makecell{0.050/\\\texttt{SCS}}}\\
                              & & & 4 & 100 & 0 & 0 & 0.086 \\
                              & & & 5 & 100 & 0 & 0 & 0.103 \\
        \hline
        \multirow{3}*{(15,20)} & \multirow{3}*{2} & \multirow{3}*{36} 
                               & 3 & 100 & 0 & 0 & 0.227 &
        \multirow{3}*{\makecell{6/\\\texttt{CSDP}}} & \multirow{3}*{\makecell{0.063/\\\texttt{CSDP}}}\\
                              & & & 4 & 100 & 0 & 0 & 0.232 \\
                              & & & 5 & 100 & 0 & 0 & 0.279 \\
        \hline
        \multirow{3}*{(30,40)} & \multirow{3}*{2} & \multirow{3}*{71} 
                               & 3 & 100 & 0 & 0 & 2.883 &
        \multirow{3}*{\makecell{13/\\\texttt{Hypatia,CSDP}}} & \multirow{3}*{\makecell{0.144/\\\texttt{CSDP}}}\\
                              & & & 4 & 100 & 0 & 0 & 1.213 \\
                              & & & 5 & 100 & 0 & 0 & 1.431 \\
        \hline
        \multirow{3}*{(50,60)} & \multirow{3}*{2} & \multirow{3}*{111} 
                               & 3 & 100 & 0 & 0 & 17.905 &
        \multirow{3}*{\makecell{21/\\\texttt{Hypatia,CSDP}}} & \multirow{3}*{\makecell{0.431/\\\texttt{CSDP}}}\\
                              & & & 4 & 100 & 0 & 0 & 3.417 \\
                              & & & 5 & 100 & 0 & 0 & 4.123 \\
        \hline
        \multirow{3}*{(70,80)} & \multirow{3}*{2} & \multirow{3}*{151} 
                               & 3 & 100 & 0 & 0 & 53.350 &
        \multirow{3}*{\makecell{29/\\\texttt{Hypatia,CSDP}}} & \multirow{3}*{\makecell{1.149/\\\texttt{CSDP}}}\\
                              & & & 4 & 100 & 0 & 0 & 8.116 \\
                              & & & 5 & 100 & 0 & 0 & 8.962 \\
        \hline
        \multirow{3}*{(5,10,15)} & \multirow{3}*{3} & \multirow{3}*{32} 
                                 & 4 & 100 & 0 & 0 & 0.283 &
        \multirow{3}*{\makecell{6/\\\texttt{Hypatia}}} & \multirow{3}*{\makecell{0.080/\\\texttt{SCS}}}\\
                              & & & 5 & 100 & 0 & 0 & 0.270 \\
                              & & & 6 & 100 & 0 & 0 & 0.311 \\
        \hline
        \multirow{3}*{(10,20,30)} & \multirow{3}*{3} & \multirow{3}*{62} 
                                  & 4 & 100 & 0 & 0 & 6.120 &
        \multirow{3}*{\makecell{11/\\\texttt{Hypatia,CSDP}}} & \multirow{3}*{\makecell{0.421/\\\texttt{SCS}}}\\
                              & & & 5 & 100 & 0 & 0 & 1.225 \\
                              & & & 6 & 100 & 0 & 0 & 1.373 \\
        \hline
        \multirow{3}*{(20,30,40)} & \multirow{3}*{3} & \multirow{3}*{92} 
                                  & 4 & 100 & 0 & 0 & 22.627 &
        \multirow{3}*{\makecell{14/\\\texttt{Hypatia}}} & \multirow{3}*{\makecell{0.852/\\\texttt{CSDP}}}\\
                              & & & 5 & 100 & 0 & 0 & 3.152 \\
                              & & & 6 & 100 & 0 & 0 & 3.355 \\
        \hline
        \multirow{3}*{(5,10,15,20)} & \multirow{3}*{4} & \multirow{3}*{53} 
                                    & 5 & 100 & 0 & 0 & 4.930 &
        \multirow{3}*{\makecell{9/\\\texttt{Hypatia,CSDP}}} & \multirow{3}*{\makecell{0.314/\\\texttt{SCS}}}\\
                              & & & 6 & 100 & 0 & 0 & 1.083 \\
                              & & & 7 & 100 & 0 & 0 & 1.178 \\
        \hline
        \multirow{3}*{(10,15,20,25)} & \multirow{3}*{4} & \multirow{3}*{73} 
                                     & 5 & 100 & 0 & 0 & 16.326 &
        \multirow{3}*{\makecell{12/\\\texttt{Hypatia,CSDP}}} & \multirow{3}*{\makecell{0.585/\\\texttt{CSDP}}}\\
                              & & & 6 & 100 & 0 & 0 & 2.392 \\
                              & & & 7 & 100 & 0 & 0 & 2.408 \\
        \hline
        \multirow{3}*{(15,20,25,30)} & \multirow{3}*{4} & \multirow{3}*{93} 
                                     & 5 & 100 & 0 & 0 & 42.774 &
        \multirow{3}*{\makecell{15/\\\texttt{Hypatia,CSDP}}} & \multirow{3}*{\makecell{0.933/\\\texttt{CSDP}}}\\
                              & & & 6 & 100 & 0 & 0 & 4.488 \\
                              & & & 7 & 100 & 0 & 0 & 4.267 \\
        \hline
        \multirow{3}*{(5,10,15,20,25)} & \multirow{3}*{5} & \multirow{3}*{78} 
                                       & 6 & 100 & 0 & 0 & 37.683 &
        \multirow{3}*{\makecell{13/\\\texttt{Hypatia,CSDP}}} & \multirow{3}*{\makecell{0.404/\\\texttt{CSDP}}}\\
                              & & & 7 & 100 & 0 & 0 & 4.294 \\
                              & & & 8 & 100 & 0 & 0 & 3.738 \\
        \hline
    \end{tabular}
    \caption{Some experiment results for rational normal scrolls}
    \label{tab:ScrollResults}
\end{table}

In Table~\ref{tab:ScrollResults}, in the left part ``Scroll Data,'' we specify the heights \((n_1,\dots,n_m)\) and the dimension \(m\) of the scroll, together with the dimension \(n\) of the ambient projective space and the number of squares \(k\) for our nonconvex rank-\(k\) formulation in the first columns.
In the middle part ``Low-rank Form.,'' the column ``Succ.'' shows the number of experiment runs that the LBFGS algorithm returns a tuple of linear forms \(\bdl\), such that the distance between its sum of squares and the target \(\nVert{\sigma_k(\bdl)-\bar{f}}\) is no more than the preset threshold \(\epsilon:=10^{-6}\);
the column ``Spur.'' shows the number of experiment runs where the LBFGS algorithm has converged (according to the default settings in \texttt{NLopt} package) but the distance \(\nVert{\sigma_k(\bdl)-\bar{f}}>\epsilon\);
the column ``Unfin.'' shows the number of experiment runs with neither of the above outcomes, i.e., the LBFGS algorithm has not converged within the preset time limits (600 seconds) or the preset maximum number of \(\sigma_k\) and its differential evaluations (\(20n\) evaluations).
The column ``Time'' shows the mean computational times of the experiment runs where the LBFGS algorithm has converged.
Moreover, in the right part ``SDP Form.,'' we report the minimum rank of the solutions found by the SDP formulation~\eqref{eq:SDPMinTrace}, together with the corresponding solver(s) in the column ``Min.\ Rank/Solver.''
The minimum computational time for the SDP formulation, together with the corresponding solver, is included in the last column ``Min.\ Time (s)/Solver.''
From these results, we observe that 
\begin{itemize}
    \item the LBFGS algorithm can successfully converge to a global solution within the computational budget in all experiment runs; 
    \item using a number of squares \(k\) larger than the Pythagoras number may reduce the computation time of the LBFGS method (by over 90\% in the cases of \((15,20,25,30)\)- and \((5,10,15,20,25)\)-scrolls); and
    \item the ranks of the solution found by the SDP formulation~\eqref{eq:SDPMinTrace} seem to grow with \(n\), in contrast with the low-rank formulation~\eqref{eq:SOSMinimizationProblem}.
\end{itemize}
It is noted, to our surprise, that on all the scroll instances presented in Table~\ref{tab:ScrollResults}, the minimum computational time via the SDP formulation is usually shorter than our low-rank formulation.
We thus further compare the computational efficiency of the low-rank formulation and the SDP formulation by focusing on surface scrolls (\(m=2\)) with large projective space dimension \(n\), where we consider heights \((n_1,n_2)\in\{(50,100),(100,200),(200,300),(300,400),(500,600),(700,800)\}\).
To save efforts, we only solve the SDP formulation with \texttt{CSDP} as it is the fastest on the larger instances in Table~\ref{tab:ScrollResults}.
We also remove the limits on the computational time and number of evaluations.
The results for 5 independent repetitions are plotted with logarithmic scales in Figure~\ref{fig:ComputationTimeSurfaceScroll}.
While the SDP formulation (solved by an IP method) can be faster on smaller instances, our low-rank formulation (solved by an LBFGS method) scales better on larger instances.

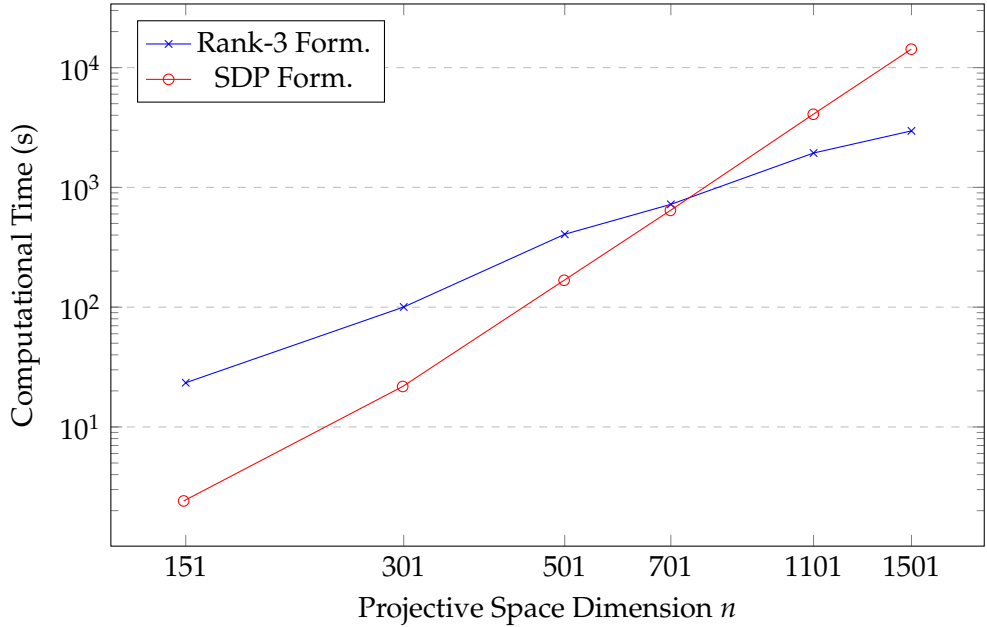
\begin{figure}[htbp]
    \centering
    \begin{tikzpicture}
    \begin{loglogaxis}[
        width=0.75\textwidth,
        height=0.5\textwidth,
        xtick={151,301,501,701,1101,1501},
        xticklabels={151,301,501,701,1101,1501},
        xlabel={Projective Space Dimension $n$},
        ylabel={Computational Time (s)},
        legend pos=north west,
        ymajorgrids=true,
        grid style=dashed,
    ]
    \addplot[
        color=blue,
        mark=x,
        ]
        coordinates {
            (151,23.39)
            (301,100.33)
            (501,406.10)
            (701,722.04)
            (1101,1936.63)
            (1501,2963.91)
        };
        \addlegendentry{Rank-3 Form.}
    \addplot[
        color=red,
        mark=o,
        ]
        coordinates {
            (150,2.41)
            (300,21.77)
            (500,167.64)
            (700,642.54)
            (1100,4079.39)
            (1500,14260.07)
        };
        \addlegendentry{SDP Form.}
    \end{loglogaxis}
    \end{tikzpicture}
    \caption{Median Computational Times for Surface Scrolls (Axes on Log Scales)}
    \label{fig:ComputationTimeSurfaceScroll}
\end{figure}

The second experiment is about plane curves.
Let \(Y\subset\bbP^2\) be a plane curve defined by a cubic homogeneous polynomial.
To certify nonnegativity of a degree-\(2d\) polynomial on the curve, we can consider the embedded curve \(X=\nu_d(Y)\subset\operatorname{span}(X)\subset\bbP^{(d+2)(d+1)/2}\), where \(\operatorname{span}(X)\) is a subspace of \(\bbP^{(d+2)(d+1)/2}\) of dimension \(\binom{d+2}{d}-\binom{d-3+2}{d-3}-1=3d-1\), so \(\deg{X}=3d>\codim{X}+1=3d-1\).
Thus \(X\) is a variety of almost minimal degree in its span, but not a variety of minimal degree.
Nevertheless, the real gonality of \(Y\) ensures that \(\Pythagoras(X)=3\)~\cite{blekherman2021sums}.
We conduct the experiments in the following way:
first we randomly generate the integer monomial coefficients (between \(-7\) and \(7\)) of the cubic defining \(Y\);
then we randomly generate \(\bar{f}\in R_2\) and then \(\bdl^{\rm init}\in R_1^k\) for each \(k\in\{3,4,5\}\) using standard normal distributions as we did in the experiments on rational normal scrolls.

The results for 100 independent experiment runs are then summarized in Table~\ref{tab:PlaneCubicResults}, where ``Cubic'' is the defining polynomial, ``Degree'' represents the degree \(d\), while the other columns ``Succ.,'' ``Unfin.,'' ``Spur.,'' and ``Time'' are the same as those in Table~\ref{tab:ScrollResults}.
Despite that \(X\) is not of minimal degree, we do not encounter any case where the LBFGS algorithm converges to a spurious stationary point.
The mean computational time grows with the number of squares \(k\) we use, for every fixed cubic curve and target degree \(d\).
This is different from the results we got in Table~\ref{tab:ScrollResults} and suggests that it is not always faster to use more squares.

\begin{table}[htbp]
    \centering
    \begin{tabular}{cccc|cccc}
        \hline
        Cubic & Degree & $n$ & $\;k\;$ & Succ. & Unfin. & Spur. & Time (s) \\
        \hline
        \multirow{15}*{\parbox{6cm}{$-3x_0^3+x_0^2x_1-x_0x_1^2-6x_1^3+3x_0^2x_2+6x_0x_1x_2+5x_1^2x_2+5x_0x_2^2+5x_1x_2^2+3x_2^3$}} 
         & \multirow{3}*{10} & \multirow{3}*{30}
         &  3 & 100 & 0 & 0 & 0.264 \\
         & & & 4 & 100 & 0 & 0 & 0.329 \\
         & & & 5 & 100 & 0 & 0 & 0.404 \\
         \cline{2-8}
         & \multirow{3}*{20} & \multirow{3}*{60} 
         &  3 & 100 & 0 & 0 & 1.869 \\
         & & & 4 & 100 & 0 & 0 & 2.421 \\
         & & & 5 & 100 & 0 & 0 & 3.003 \\
         \cline{2-8}
         & \multirow{3}*{30} & \multirow{3}*{90} 
         &  3 & 100 & 0 & 0 & 8.432 \\
         & & & 4 & 100 & 0 & 0 & 10.931 \\
         & & & 5 & 100 & 0 & 0 & 13.674 \\
         \cline{2-8}
         & \multirow{3}*{40} & \multirow{3}*{120} 
         &  3 & 100 & 0 & 0 & 12.204 \\
         & & & 4 & 100 & 0 & 0 & 15.945 \\
         & & & 5 & 100 & 0 & 0 & 19.860 \\
         \cline{2-8}
         & \multirow{3}*{50} & \multirow{3}*{150}
         &  3 & 100 & 0 & 0 & 24.242 \\
         & & & 4 & 100 & 0 & 0 & 31.905 \\
         & & & 5 & 100 & 0 & 0 & 39.570 \\
        \hline
        \multirow{15}*{\parbox{6cm}{$5x_0^3-2x_0^2x_1-x_0x_1^2+3x_1^3-x_0^2x_2+7x_0x_1^2-x_1^2x_2-5x_0x_2^2-5x_1x_2^2-2x_2^3$}} 
         & \multirow{3}*{10} & \multirow{3}*{30}
         &  3 & 100 & 0 & 0 & 0.143 \\
         & & & 4 & 100 & 0 & 0 & 0.171 \\
         & & & 5 & 100 & 0 & 0 & 0.210 \\
         \cline{2-8}
         & \multirow{3}*{20} & \multirow{3}*{60} 
         &  3 & 100 & 0 & 0 & 0.938 \\
         & & & 4 & 100 & 0 & 0 & 1.225 \\
         & & & 5 & 100 & 0 & 0 & 1.536 \\
         \cline{2-8}
         & \multirow{3}*{30} & \multirow{3}*{90} 
         &  3 & 100 & 0 & 0 & 3.175 \\
         & & & 4 & 100 & 0 & 0 & 4.141 \\
         & & & 5 & 100 & 0 & 0 & 5.203 \\
         \cline{2-8}
         & \multirow{3}*{40} & \multirow{3}*{120} 
         &  3 & 100 & 0 & 0 & 7.596 \\
         & & & 4 & 100 & 0 & 0 & 9.843 \\
         & & & 5 & 100 & 0 & 0 & 12.150 \\
         \cline{2-8}
         & \multirow{3}*{50} & \multirow{3}*{150}
         &  3 & 100 & 0 & 0 & 12.009 \\
         & & & 4 & 100 & 0 & 0 & 15.340 \\
         & & & 5 & 100 & 0 & 0 & 18.458 \\
        \hline
        \multirow{15}*{\parbox{6cm}{$-7x_0^3+3x_0^2x_1-5x_0x_1^2+5x_1^3-6x_0^2x_2-4x_0x_1x_2-6x_1^2x_2-6x_0x_2^2-7x_1x_2^2-2x_2^3$}} 
         & \multirow{3}*{10} & \multirow{3}*{30}
         &  3 & 100 & 0 & 0 & 0.185 \\
         & & & 4 & 100 & 0 & 0 & 0.219 \\
         & & & 5 & 100 & 0 & 0 & 0.268 \\
         \cline{2-8}
         & \multirow{3}*{20} & \multirow{3}*{60} 
         &  3 & 100 & 0 & 0 & 1.160 \\
         & & & 4 & 100 & 0 & 0 & 1.506 \\
         & & & 5 & 100 & 0 & 0 & 1.891 \\
         \cline{2-8}
         & \multirow{3}*{30} & \multirow{3}*{90} 
         &  3 & 100 & 0 & 0 & 3.845 \\
         & & & 4 & 100 & 0 & 0 & 5.012 \\
         & & & 5 & 100 & 0 & 0 & 6.226 \\
         \cline{2-8}
         & \multirow{3}*{40} & \multirow{3}*{120} 
         &  3 & 100 & 0 & 0 & 9.110 \\
         & & & 4 & 100 & 0 & 0 & 11.759 \\
         & & & 5 & 100 & 0 & 0 & 14.490 \\
         \cline{2-8}
         & \multirow{3}*{50} & \multirow{3}*{150}
         &  3 & 100 & 0 & 0 & 14.042 \\
         & & & 4 & 100 & 0 & 0 & 18.065 \\
         & & & 5 & 100 & 0 & 0 & 21.885 \\
        \hline
    \end{tabular}
    \caption{Some experiment results for plane cubic curves}
    \label{tab:PlaneCubicResults}
\end{table}

To further study the growth of computational times with respect to the degree \(d\), we fix our plane cubic curve to be \(Y:=V(-x_1^2x_2-4x_0x_1^2+6x_1x_2^2-7x_0x_1x_2+7x_0^2x_1-2x_2^3+3x_0x_2^2+x_0^2x_2-x_0^3)\subset\bbP^2\), and consider \(d\in\{100,200,300,400,500\}\). 
The results are plotted in Figure~\ref{fig:ComputationTimeCubicCurve} with logarithmic scales for the axes, where the rank-3 formulation outperforms the SDP formulation as \(d\) grows beyond \(200\).
This reconfirms the popular belief that the low-rank formulation could scale better than the SDP formulation and justifies our study of it being applied to sum-of-squares problems~\eqref{eq:SOSMinimizationProblem}.

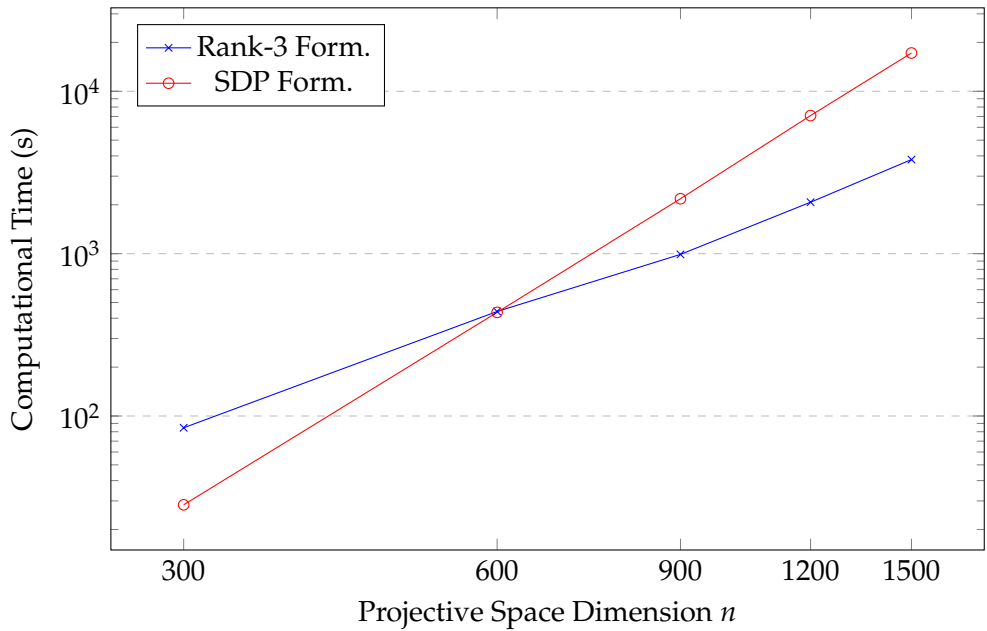
\begin{figure}[htbp]
    \centering
    \begin{tikzpicture}
    \begin{loglogaxis}[
        width=0.75\textwidth,
        height=0.5\textwidth,
        xtick={300,600,900,1200,1500},
        xticklabels={300,600,900,1200,1500},
        xlabel={Projective Space Dimension $n$},
        ylabel={Computational Time (s)},
        legend pos=north west,
        ymajorgrids=true,
        grid style=dashed,
    ]
    \addplot[
        color=blue,
        mark=x,
        ]
        coordinates {
            (300,84.63)
            (600,440.31)
            (900,989.56)
            (1200,2074.40)
            (1500,3797.04)
        };
        \addlegendentry{Rank-3 Form.}
    \addplot[
        color=red,
        mark=o,
        ]
        coordinates {
            (300,28.38)
            (600,434.83)
            (900,2179.12)
            (1200,7084.56)
            (1500,17209.14)
        };
        \addlegendentry{SDP Form.}
    \end{loglogaxis}
    \end{tikzpicture}
    \caption{Median Computational Times for Forms on a Plane Cubic Curve (Axes on Log Scales)}
    \label{fig:ComputationTimeCubicCurve}
\end{figure}

Our last experiments are conducted on Veronese varieties \(X=\nu_d(\bbP^m)\subset\bbP^n\), where \(n=\binom{d+m}{d}-1\).
Other than the case \(d=m=2\) (the Veronese surface), \(X\) is not a variety of minimal degree, and the Pythagoras number \(\Pythagoras(X)\) is not known precisely, unlike the case of plane cubic curves.
To circumvent this issue, we adopt an upper bound on \(\Pythagoras(X)\), which is
the smallest integer \(\bar{k}\) such that \(\binom{\bar{k}+1}{2}\ge\dim_\bbR(R_2)\)~\cite{blekherman2021sums}.
To study the behaviour of the LBFGS algorithm with respect to different numbers of squares, we consider \(k=\bar{k}\), \(\lceil1.1\cdot\bar{k}\rceil\), and \(\lceil1.2\cdot\bar{k}\rceil\) for comparison.
All other experiment settings are the same as those in the rational normal scrolls and plane cubic curves.
The results are summarized in Table~\ref{tab:VeroneseResults}, from which we can see that the LBFGS algorithm is still able to converge to the global minimum in all cases without being trapped at any spurious stationary points.
This indicates that the spurious local minimum we constructed in Example~\ref{ex:QuarticsSpuriousMinima} is likely rare and does not affect the success of LBFGS algorithms on generic instances.

\begin{table}[htbp]
    \centering
    \begin{tabular}{cccc|cccc}
        \hline
        $m$ & $d$ & $n$ & $\;k\;$ & Succ. & Unfin. & Spur. & Time (s) \\
        \hline
        \multirow{3}*{4}  & \multirow{3}*{2} & \multirow{3}*{14} 
                          & 12 & 100 & 0 & 0 & 0.049 \\
                          & & & 14 & 100 & 0 & 0 & 0.048 \\
                          & & & 15 & 100 & 0 & 0 & 0.051 \\
        \hline
        \multirow{3}*{6}  & \multirow{3}*{2} & \multirow{3}*{27} 
                          & 20 & 100 & 0 & 0 & 0.382 \\
                          & & & 22 & 100 & 0 & 0 & 0.407 \\
                          & & & 24 & 100 & 0 & 0 & 0.439 \\
        \hline
        \multirow{3}*{8}  & \multirow{3}*{2} & \multirow{3}*{44} 
                          & 31 & 100 & 0 & 0 & 2.483 \\
                          & & & 35 & 100 & 0 & 0 & 2.687 \\
                          & & & 38 & 100 & 0 & 0 & 2.910 \\
        \hline
        \multirow{3}*{10} & \multirow{3}*{2} & \multirow{3}*{65} 
                          & 45 & 100 & 0 & 0 & 9.641 \\
                          & & & 50 & 100 & 0 & 0 & 10.755 \\
                          & & & 54 & 100 & 0 & 0 & 11.644 \\
        \hline
        \multirow{3}*{4}  & \multirow{3}*{3} & \multirow{3}*{34} 
                          & 20 & 100 & 0 & 0 & 0.534 \\
                          & & & 22 & 100 & 0 & 0 & 0.584 \\
                          & & & 24 & 100 & 0 & 0 & 0.636 \\
        \hline
        \multirow{3}*{3}  & \multirow{3}*{4} & \multirow{3}*{34} 
                          & 18 & 100 & 0 & 0 & 0.428 \\
                          & & & 20 & 100 & 0 & 0 & 0.471 \\
                          & & & 22 & 100 & 0 & 0 & 0.518 \\
        \hline
        \multirow{3}*{2}  & \multirow{3}*{5} & \multirow{3}*{20} 
                          & 11 & 100 & 0 & 0 & 0.065 \\
                          & & & 13 & 100 & 0 & 0 & 0.077 \\
                          & & & 14 & 100 & 0 & 0 & 0.082 \\
        \hline
    \end{tabular}
    \caption{Some experiment results for Veronese varieties}
    \label{tab:VeroneseResults}
\end{table}

\section*{Acknowledgements}
Grigoriy Blekherman was partially supported by NSF DMS Grant 1901950.
Mauricio Velasco is partially supported by Fondo Clemente Estable grant FCE-1-2023-1-176172 (ANII, Uruguay). 
Shixuan Zhang was supported by NSF DMS Grant 1929284 while he was in residence at the Institute for Computational and Experimental Research in Mathematics in Providence, RI.

\bibliographystyle{alpha}
\bibliography{ref}

\end{document}